\documentclass[10pt, leqno]{article}
\usepackage{amsmath,amsfonts,amsthm,amscd,amssymb}
\numberwithin{equation}{section}
\theoremstyle{plain}
\newtheorem{theo}{Theorem}[section]
\newtheorem{prop}[theo]{Proposition}
\newtheorem{coro}[theo]{Corollary} 
\newtheorem{lemm}[theo]{Lemma}
\theoremstyle{definition}
\newtheorem{defi}[theo]{Definition}
\newtheorem{rema}[theo]{Remark}
\newtheorem{theo-defi}[theo]{Theorem-Definition}
\newtheorem{prop-defi}[theo]{Proposition-Definition}
\newtheorem{rema-defi}[theo]{Remark-Definition}
\newtheorem{exem-defi} [theo]{Example-Definiton}
\newtheorem{prob}[theo]{Problem}

\newtheorem{exem}[theo]{Example}

\def \al{\alpha}
\def \bet{\beta}
\def \bul{\bullet}
\def \col{\colon}
\def \Del{\Delta}

\def \del{\delta}

\def \gam{\gamma}
\def \inf{\infty}

\def \kap{\kappa}

\def \lam{\lambda}

\def \Lo{\Longrightarrow}
\def \lo{\longrightarrow}
\def \lom{\longmapsto}
\def \mab{\mathbb}

\def \Om{\Omega}
\def \om{\omega}
\def \ol{\overline}
\def \os{\overset}
\def \parno{\par\noindent}

\def \sig{\sigma}

\def \sus{\subset}

\def \us{\underset}

\def \vpl{\varprojlim}

\bigskip

\newcommand{\getsfrom}
{\ensuremath{\longleftarrow\kern-.
52em\lower-.1ex\hbox%
{$\shortmid\,$}}}

\begin{document}

\title{Congruences of the cardinalities of rational points of log Fano varieties 
and log Calabi-Yau varieties over the log points of finite fields}
\author{Yukiyoshi Nakkajima 
\date{}\thanks{ 
The first named author is supported from JSPS
Grant-in-Aid for Scientific Research (C)
(Grant No.~80287440). 
The second named author is supported by 
JSPS Fellow (Grant No.~15J05073).\endgraf}}
\maketitle

$${\bf Abstract}$$
In this article we give the definitions of 
log Fano varieties and log Calabi-Yau varieties in 
the framework of theory of log schemes of Fontain-Illusie-Kato and 
give congruences of the cardinalities of rational points of them
over the log points of finite fields. 

\section{Introduction}\label{sec:int} 
In this article we discuss a new topic--rational points of 
the underlying schemes of log schemes in the sense of 
Fontaine-Illusie-Kato over the log point of 
a finite field--for interesting log schemes. 
First let us recall results on rational points of 
(proper smooth) schemes over a finite field. 
\par 
The following is famous Ax' and Katz' theorem: 

\begin{theo}[{\bf \cite{ax}, \cite{kat}}]\label{theo:axka}
Let ${\mab F}_q$ be the finite field with $q=p^e$-elements, 
where $p$ is a prime number.  
Let $n$ and $r$ be positive integers. 
Let $D_i$ $(1\leq i\leq r)$ be a hypersurface of ${\mab P}^n_{{\mab F}_q}$ 
of degree $d_i$. 
If $\sum_{i=1}^rd_i\leq n$, 
then $\# (\bigcap_{i=1}^rD_i)({\mab F}_{q^k})\equiv 1~{\rm mod}~q^k$.  
\end{theo}

\par 
In \cite{es} Esnault has proved the following theorem generalizing 
this theorem in the case where $\bigcap_{i=1}^rD_i$ is smooth over ${\mab F}_q$ 
and geometrically connected: 

\begin{theo}[{\bf \cite[Corollary 1.3]{es}}]\label{theo:esn} 
Let $X$ be a geometrically connected projective smooth scheme 
over ${\mab F}_q$. 
If $X/{\mab F}_q$ is a Fano variety $($i.~e., the inverse of the canonical sheaf 
$\om_{X/{\mab F}_q}^{-1}$ of $X/{\mab F}_q$ is ample$)$,  
then $\# X({\mab F}_{q^k}) \equiv 1~{\rm mod}~ q^k$ $(k\in {\mab Z}_{\geq 1})$.
\end{theo} 

In \cite{ki} Kim has proved the following theorem 
and he has reproved Esnault's theorem as a corollary of his theorem 
by using the Lefschetz trace formula for 
the crystalline cohomology of $X/{\mab F}_q$: 

\begin{theo}[{\bf \cite[Theorem 1]{ki}}]\label{theo:kim} 
Let $\kap$ be a perfect field of characteristic $p>0$. 
Set ${\cal W}:={\cal W}(\kap)$ and $K_0:={\rm Frac}({\cal W})$. 
Let $X$ be a  projective smooth scheme over $\kap$. 
If $X/\kap$ is a Fano variety,  
then $H^i(X,{\cal W}({\cal O}_X))\otimes_{\cal W}{K_0} =0$ for $i>0$.
\end{theo}

In  \cite{gir} Gongyo, Nakamura and Tanaka 
have proved the following theorem generalizing 
(\ref{theo:esn}) for the 3-dimensional case 
by using methods of MMP(=minimal model program) 
in characteristic $p\geq 7$: 

\begin{theo}[{\bf \cite[Theorem (1.2), (1.3)]{gir}}]\label{theo:gnt} 
Let $\kap$ be as in {\rm (\ref{theo:kim})}. 
Assume that $p\geq 7$. 
Let $X$ be a geometrically connected proper variety over $\kap$.
Let $\Del$ be an effective ${\mab Q}$-Cartier divisor on $X$. 
Assume that $(X,\Del)$ is klt$($=Kawamata log terminal$)$ pair over $\kap$ 
and that $-(K_X+\Del)$ is a ${\mab Q}$-Cartier ample divisor on $X$, 
where $K_X$ is the canonical divisor on $X$. 
Then the following hold$:$ 
\par 
$(1)$  $H^i(X,{\cal W}({\cal O}_{X}))\otimes_{\cal W}{K_0}=0$ for $i>0$. 
\par 
$(2)$ Assume that $\kap={\mab F}_q$. Then 
$\# X({\mab F}_{q^k}) \equiv 1~{\rm mod}~ q^k$ $(k\in {\mab Z}_{\geq 1})$.
\end{theo}  

\parno 
See \cite{nt} for the case where $-(K_X+\Del)$ is nef and big 
and $(X,\Del)$ is log canonical. 
\par 
In this article we give other generalizations of the Theorems 
(\ref{theo:esn}) and (\ref{theo:kim})  
under the assumption of certain finiteness: 
we give the definition of a log Fano variety and 
we prove a log and stronger version (\ref{theo:vc}) below of Kim's theorem 
under the assumption 
as a really immediate good application of a recent result:  
Nakkajima-Yobuko's Kodaira vanishing theorem for 
a quasi-$F$-split projective log smooth scheme of vertical type (\cite{ny}). 
In this vanishing theorem, 
we use theory of log structures due to Fontaine-Illusie-Kato 
(\cite{klog1}, \cite{klog2}) essentially. 
(See \S\ref{sec:psso} for the precise statement of this vanishing theorem.) 
As a corollary of (\ref{theo:vc}), we obtain the  
congruence of the cardinality of rational points of a log Fano variety
over the log point of ${\mab F}_q$ ((\ref{coro:npf}) below).  
\par 
To state our result (\ref{theo:vc}),  
we first recall the notion of the quasi-Frobenius splitting height due to Yobuko, 
which plays an important role for log Fano varieties in this article. 
\par 
Let $Y$ be a scheme of characteristic $p>0$. 
Let $F_Y\col Y\lo Y$ be the Frobenius endomorphism of $Y$. 
Set $F:={\cal W}_n(F_Y^*)\col 
{\cal W}_n({\cal O}_Y)\lo F_{Y*}({\cal W}_n({\cal O}_Y))$. 
This is a morphism of ${\cal W}_n({\cal O}_Y)$-modules. 
In \cite{y} Yobuko has introduced the notion of the 
quasi-Frobenius splitting height $h^F(Y)$ for $Y$. 
(In [loc.~cit.] he has denoted it by ${\rm ht}^S(Y)$.) 
It is the minimum of positive integers $n$'s such that 
there exists a morphism 
$\rho \col F_{Y*}({\cal W}_n({\cal O}_Y))\lo {\cal O}_Y$ 
of ${\cal W}_n({\cal O}_Y)$-modules 
such that 
$\rho \circ F\col {\cal W}_n({\cal O}_Y)\lo {\cal O}_Y$ 
is the natural projection. 
(If there does not exist such $n$, then  we set $h^F(Y)=\infty$.)  
This is a highly nontrivial generalization of the notion of the Frobenius splitting 
by Mehta and Ramanathan in \cite{mr} 
because they have said that, for a scheme $Z$ of characteristic $p>0$, 
$Z$ is a Frobenius splitting(=$F$-split) scheme if 
$F\col {\cal O}_Z\lo F_{Z*}({\cal O}_Z)$ has a section of ${\cal O}_Z$-modules. 
Because the terminology ``quasi Frobenius splitting height'' is too long, 
we call this {\it Yobuko height}. 
\par 
Let $\kap$ be a perfect field of characteristic $p>0$. 
Let $s$ be a log scheme whose underlying scheme 
is ${\rm Spec}(\kap)$ and whose log structure 
is associated to a morphism ${\mab N}\owns 1\lom a\in \kap$ 
for some $a\in \kap$. That is, $s$ is the log point of $\kap$ or 
$({\rm Spec}(\kap),\kap^*)$. 
Let $X/s$ be a proper (not necessarily projective) 
log smooth scheme of pure dimension $d$ 
of vertical type with log structure 
$(M_X,\al \col M_X\lo {\cal O}_X)$. 
Here ``vertical type'' means that 
$\al({\cal I}_{X/s}){\cal O}_X={\cal O}_X$, 
where ${\cal I}_{X/s}$ is Tsuji's ideal sheaf of the log structure $M_X$ of $X$ 
defined in \cite{tsp} and denoted by $I_f$ in [loc.~cit.], 
where $f\col X\lo s$ is the structural morphism.  
(In \S\ref{sec:psso} below we recall the definition of ${\cal I}_{X/s}$.) 
For example, the product of (locally) simple normal crossing log schemes 
over $s$ defined in \cite{nlk3}, \cite{ny} and \cite{nlw} is of vertical type. 
Let $\os{\circ}{X}$ be the underlying scheme of $X$. 
Let $\Om^i_{X/s}$ be the sheaf of log differential forms of degree $i$ on $\os{\circ}{X}$, 
which has been denoted by $\om^i_{X/s}$ in \cite{klog1}. 
Set $\om_{X/s}:=\Om^d_{X/s}$. 
We say that $X/s$ is a {\it log Fano scheme} 
if $\om_{X/s}^{-1}$ is ample.  
Moreover, if $\os{\circ}{X}$ is geometrically connected, then 
we say that $X/s$ a {\it log Fano variety}.
\par 
In this article we prove the following: 

\begin{theo}\label{theo:vc} 
Let $X/s$ be a log Fano scheme. 
Assume that $h^F(\os{\circ}{X})<\infty$. 
Then $H^i(X,{\cal W}_n({\cal O}_X))=0$ for $i>0$ and for $n>0$. 
Consequently $H^i(X,{\cal W}({\cal O}_X))=0$ for $i>0$.
\end{theo}

As mentioned above, 
we obtain this theorem immediately by using  
Nakkajima-Yobuko's Kodaira vanishing theorem for 
a quasi-$F$-split projective log smooth scheme of 
vertical type (\cite{ny}).  
As a corollary of this theorem, 
we obtain the following: 

\begin{coro}\label{coro:npf}
Let $X/s$ be a log Fano variety. 
Assume that $\kap={\mab F}_q$ and 
that $h^F(\os{\circ}{X})<\infty$. 
Then  
\begin{align*}  
\# \os{\circ}{X}({\mab F}_{q^k}) \equiv 1~{\rm mod}~ q^k \quad (k\in {\mab Z}_{\geq 1}). 
\tag{1.6.1}\label{ali:xfq} 
\end{align*}
In particular $\os{\circ}{X}({\mab F}_q)\not= \emptyset$. 
\end{coro} 

This is a generalization of Esnault's theorem (\ref{theo:esn})
under the assumption of the finiteness of the Yobuko height. 
To derive (\ref{coro:npf}) from (\ref{theo:vc}), 
we use 
\smallskip 
\par 
(A): \'{E}tess-Le Stum's Lefschetz trace formula for rigid cohomology   
(with compact support) (\cite{el}) 
\medskip 
\parno 
and 
\medskip 
\par 
(B) Berthelot-Bloch-Esnault's calculation of the slope $< 1$-part of the rigid cohomology 
(with compact support) via Witt sheaves (\cite{bbe})
\medskip 
\parno
as in \cite{bbe}, \cite{gir} and \cite{nt}. 
However our proofs of (\ref{theo:vc}) and (\ref{coro:npf}) are 
very different from Esnault's, Kim's 
and Gongyo-Nakamura-Tanaka's proofs of 
(\ref{theo:esn}), (\ref{theo:kim}) and (\ref{theo:gnt}) 
in their articles because 
we do not use the rational connectedness of 
a Fano variety which has been used in them. 

\par 
We guess that the assumption of the finiteness of the 
Yobuko height
is not a strong one for log Fano schemes. 
However this assumption is not always satisfied for smooth Fano 
schemes because the Kodaira vanishing holds if the Yobuko height
is finite and because the Kodaira vanishing does not hold for 
certain Fano varieties (\cite{lr}, \cite{hl}, \cite{to}); 
the Yobuko heights  
of them are infinity. 
Hence to calculate the Yobuko heights 
of (log) Fano schemes is a very interesting problem.  
\par 
The conclusion of (\ref{coro:npf}) holds for a proper scheme $Y/{\mab F}_q$ 
such that $H^i(Y,{\cal O}_Y)=0$ $(i>0)$. 
H.~Tanaka has kindly told me that it is not known whether there exists an example 
of a smooth Fano variety over $\kap$ for 
which this vanishing of the cohomologies does not hold. 
(In \cite{j} Joshi has already pointed out this; 
Shepherd-Barron has already proved that this vanishing holds 
for a smooth Fano variety of dimension 3 (\cite[(1.5)]{sb})).



\smallskip
\par 
On the other hand, it is not clear at all that there is a precise rule as above 
about congruences of the cardinalities of the rational points 
of varieties except Fano varieties. 
One may think that there is no rule for them. 
In this article we show that this is not the case for 
log Calabi-Yau varieties over $s$ of any dimension when 
$\os{\circ}{s}={\rm Spec}({\mab F}_q)$; 
we are more interested in the cardinalities of 
the rational points of log Calabi-Yau varieties than those of 
log Fano varieties. 
\par 
First let us recall the following suggestive observation, which  
seems well-known (\cite{bth}).   
\par 
Let $E$ be an elliptic curve over ${\mab F}_p$. 
It is well-known that $E$ is nonordinary if and only if 
\begin{align*} 
\#E({\mab F}_p)=p+1
\tag{1.6.2}\label{ali:pp1}
\end{align*}  
if $p\geq 5$. 
By the purity of the weight for $E/{\mab F}_p$: 
\begin{align*} 
\vert \#E({\mab F}_p)- (p+1)\vert \leq 2\sqrt{p}, 
\tag{1.6.3}\label{ali:ppwt1}
\end{align*}
this equality is equivalent to a congruence 
\begin{align*} 
\#E({\mab F}_p)\equiv 1~{\rm mod}~p
\tag{1.6.4}\label{ali:ppwcg}
\end{align*}
since $\sqrt{p}>2$. 
\par 
In this article we generalize 
the congruence (\ref{ali:ppwcg}) 
for higher dimensional (log) varieties as follows.
(We also generalize (\ref{ali:pp1}) for 
for any nonordinary  elliptic curve over 
${\mab F}_q$ when $p\geq 5$.)    
\par 
Let $X/s$ be a proper (not necessarily projective) 
simple normal crossing log scheme of pure dimension $d$. 
Recall that, in \cite{ny}, 
we have said that $X/s$ is a log Calabi-Yau scheme of pure dimension $d$ if 
$H^i(X,{\cal O}_X)=0$ $(0< i< d)$ and $\om_{X/s}\simeq {\cal O}_X$.  
Moreover, if $\os{\circ}{X}$ is geometrically connected, then we say that 
$X/s$ is a log Calabi-Yau variety of pure dimension $d$. 
(This is a generalization of a log K3 surface defined in \cite{nlk3}.)
Note that $H^d(X,{\cal O}_X) =H^d(X,\om_{X/s})\simeq \kap$.  
The last isomorphism is obtained by log Serre duality of Tsuji 
(\cite[(2.21)]{tsp}). 
More generally, we consider a proper scheme $Y$ of 
pure dimension $d$ satisfying only 
the following four conditions: 
\medskip 
\par 
(a) $H^0(Y,{\cal O}_Y)={\kap}$,  
\par 
(b) $H^i(Y,{\cal W}({\cal O}_Y))_{K_0}=0$ for $0<i<d-1$, 
\par  
(c) $H^{d-1}(Y,{\cal O}_Y)=0$ if $d\geq 2$, 
\par  
(d) $H^d(Y,{\cal O}_Y)\simeq {\kap}$. 
\medskip 
\parno 
Let $\Phi_{Y/\kap}$
be the Artin-Mazur formal group of $Y/\kap$ in degree $d$, 
that is, $\Phi_{Y/\kap}$ is the following functor: 
$$\Phi_{Y/\kap}(A):=\Phi^d_{Y/\kap}(A):={\rm Ker}(H^d_{\rm et}
(Y\otimes_{\kap}A,{\mab G}_m)\lo 
H^d_{\rm et}(Y,{\mab G}_m)) \in ({\rm Ab})$$
for artinian local $\kap$-algebras $A$'s with residue fields $\kap$. 
Then $\Phi_{Y/\kap}$ is pro-represented by 
a commutative formal Lie group over $\kap$ (\cite{am}). 
Denote the height of $\Phi_{Y/\kap}$ by $h(Y/\kap)$.  
We prove the following$:$


\begin{theo}\label{theo:xfh}
Let $Y/\kap$ be as above. 
Assume that $\kap ={\mab F}_q$.
Set $h:=h(Y/{\mab F}_q)$. 
Then the following hold$:$ 
\par 
$(1)$ Assume that $h=\infty$. 
Then 
\begin{equation*}
\# Y({\mab F}_{q^k}) \equiv
1~{\rm mod}~ q^k \quad (k\in {\mab Z}_{\geq 1}).    
\tag{1.7.1}\label{eqn:kfdintd}
\end{equation*}
In particular, $Y({\mab F}_{q})\not= \emptyset$. 
\par 
$(2)$ Assume that $2\leq h<\infty$. 
Let $\lceil~\rceil$ 
be the ceiling function$:$  
$\lceil x \rceil:=\min\{n\in {\mab Z}~\vert~x\leq n\}$. 
Then 
\begin{equation*}
\# Y({\mab F}_{q^{k}}) \equiv1~{\rm mod}~ p^{\lceil ek(1-h^{-1})\rceil}
\quad (k\in {\mab Z}_{\geq 1}).    
\tag{1.7.2}\label{eqn:kfd2td}
\end{equation*}
In particular, $Y({\mab F}_{q})\not= \emptyset$ 
$($recall that $e=\log_pq)$.  
\par 
$(3)$ Assume that $h=1$. Then 
\begin{equation*}
\# Y({\mab F}_{q^{k}}) \not\equiv1~{\rm mod}~ p 
\quad (k\in {\mab Z}_{\geq 1}).    
\tag{1.7.3}\label{eqn:kfdbntd}
\end{equation*}
$($In particular $Y({\mab F}_{q^k})$ can be empty.$)$ 
\end{theo}  
To give the statement (\ref{theo:xfh}) is a highly nontrivial work.  
However the proof of (\ref{theo:xfh}) is not difficult. 
(It does not matter whether the proof is not difficult.)
As far as we know, (\ref{theo:xfh}) even in the 2-dimensional trivial logarithmic  
and smooth case, i.~e., the case of K3 surfaces over finite fields, 
is a new result. Even in the case $d=1$, 
$Y$ need not be assumed to be an elliptic curve over ${\mab F}_q$. 
\par
The heights of Artin-Mazur formal groups 
describe the different phenomena 
about the congruences of rational points for schemes satisfying 
four conditions (a), (b), (c) and (d).  
\par 
By using (\ref{theo:xfh}), 
we raise an important problem how 
the certain supersingular prime ideals are distributed for 
a smooth Calabi-Yau variety of dimension 
less than or equal to $2$ over a number field.   
(I think that there is no relation with Sato-Tate conjecture in non-CM cases.)
\par 
To obtain (\ref{theo:xfh}), we use the theorems 
(A) and (B) explained after (\ref{coro:npf}) again 
and the determination of the slopes of the Dieudonn\'{e} module 
$D(\Phi_{Y/\kap})$ of $\Phi_{Y/\kap}$. 

\par 
The contents of this article are as follows. 
\par 
In \S\ref{sec:pss} we recall 
\'{E}tess-Le Stum's Lefschetz trace formula for rigid cohomology, 
Berthelot-Bloch-Esnault's theorem and 
the congruence of the cardinality of  
rational points of a separated scheme of finite type over a finite field.  
\par 
In \S\ref{sec:psso} we prove 
(\ref{theo:vc}) and (\ref{coro:npf}). 
\par 
In \S\ref{sec:cy} we prove (\ref{theo:xfh}). 
We also raise the important problem about  
the distribution of supersingular primes already mentioned. 
\par 
In \S\ref{sec:tkz} we give the formulas of two kinds of zeta functions of 
a few projective SNCL(=simple normal crossing log) schemes 
over the log point of a finite field.  One kind of them gives us 
examples of the conclusions of the congruences in 
(\ref{coro:npf}) and (\ref{theo:xfh}). 
\par 
In \S\ref{rema:arkg} 
we give a remark on 
Van der Geer and Katsura's characterization of the height $h(Y/\kap)$ 
(\cite{vgk}). 

\par
\bigskip
\parno
{\bf Acknowledgment.} 
I have begun this work after listening to 
Y.~Nakamaura's very clear talk in which 
the main theorem in \cite{nt} has been explained
in the conference ``Higher dimensional algebraic geometry'' 
of Y.~Kawamata in March 2018 at Tokyo University. 
The talk of Y.~Gongyo in January 2017 at Tokyo Denki university 
for the explanation of the main theorem in \cite{gir} 
has given a very good influence to this article. 
Without their talks, I have not begun this work. 
I would like to express sincere gratitude to them.  
I would also like to express sincere thanks to 
H.~Tanaka and S.~Ejiri for their kindness for informing me  of 
the articles \cite{lr}, \cite{hl}, \cite{to} and giving me an important remark.  
\bigskip
\par\noindent
{\bf Notations.}  
(1) For an element $a$ of a commutative ring $A$ with unit element and 
for an $A$-modules $M$, $M/a$ denotes $M/aM$.   
\par 
(2) For a finite field ${\mab F}_q$, 
$s_{{\mab F}_q}$ denotes 
the log point whose underlying scheme 
is ${\rm Spec}({\mab F}_q)$. 

\section{Preliminaries}\label{sec:pss} 
In this section we recall \'{E}tess-Le Stum's Lefschetz trace formula 
for rigid cohomology with compact support (\cite{el}) and 
Berthelot-Bloch-Esnault's calculation of   
the slope $< 1$-part of the rigid cohomology 
with compact support via Witt sheaves with compact support (\cite{bbe}). 
\par 
Let $K_0({\mab F}_q)$ be the fraction field of the Witt ring 
${\cal W}({\mab F}_q)$ of ${\mab F}_q$. 
Let $Y$ be a separated scheme of finite type over ${\mab F}_q$ of dimension $d$. 
Let $F_q\col Y\lo Y$ be the $q$-th power Frobenius endomorphism of $Y$. 
The following is \'{E}tess-Le Stum's Lefschetz trace formula proved in 
\cite[Th\'{e}or\`{e}me II]{el}:  
\begin{align*} 
\# Y({\mab F}_{q})=\sum_{i=0}^{2d}(-1)^i{\rm Tr}
(F^*_q\vert H^i_{\rm rig,c}(Y/K_0({\mab F}_q))).  
\tag{2.0.1}\label{ali:trfh} 
\end{align*} 
Let $\{\alpha_{ij}\}_j$ be an eigenvalue of $F^*_q$ on 
$H^i_{\rm rig,c}(Y/K_0({\mab F}_q))$. 
Then 
\begin{align*} 
\# Y({\mab F}_q)=\sum_{i=0}^{2d}(-1)^i(\sum_j{\alpha}_{ij}).  
\tag{2.0.2}\label{ali:trfeh} 
\end{align*} 
By \cite[(3.1.2)]{clpe} (see also \cite[(17.2)]{nh3}), 
\begin{align*} 
\max \{0, i-d\}\leq {\rm ord}_q(\alpha_{ij})\leq \min \{i, d\}.
\tag{2.0.3}\label{ali:trfodh} 
\end{align*}   
Henceforth we consider the equalities (\ref{ali:trfh}) and 
(\ref{ali:trfeh}) as the equalities in the integer ring $\ol{{\cal W}({\mab F}_q)}$
of an algebraic closure of $\ol{K_0({\mab F}_q)}$.  

Let $\kap$ be a perfect field of characteristic $p>0$. 
Let $Y/\kap$ be a separated scheme of finite type. 
Let $K_0$ be the fraction field of the Witt ring ${\cal W}$ of $\kap$. 
Let 
$H^i_{\rm rig,c}(Y/K_0)_{[0,1)}$
be the slope $<1$-part of 
the rigid cohomology $H^i_{\rm rig,c}(Y/K_0)$ 
with compact support 
with respect to the absolute Frobenius endomorphism of $Y$. 
Let 
$H^i_{\rm c}(Y,{\cal W}({\cal O}_{Y,K_0}))$  
be the cohomology of the Witt sheaf 
with compact support of $Y/K_0$ defined by 
Berthelot, Bloch and Esnault in \cite{bbe}: 
\begin{align*} 
H^i_{\rm c}(Y,{\cal W}({\cal O}_{Y,K_0}))
:=H^i(Y,{\cal W}({\cal I}_{K_0})), 
\end{align*} 
where ${\cal W}({\cal I}_{K_0})
:={\rm Ker}({\cal W}({\cal O}_Z)_{K_0}\lo {\cal W}({\cal O}_Z/{\cal I})_{K_0})$  
and ${\cal I}$ is a coherent ideal sheaf of ${\cal O}_Z$ 
for an open immersion 
$Y\os{\sus}{\lo} Z$ into a proper scheme over $\kap$
such that $V({\cal I})=Z\setminus Y$.  
They have proved that $H^i_{\rm c}(Y,{\cal W}({\cal O}_{Y,K_0}))$ 
is independent of the choice of the closed immersion.
By the definition of $H^i_{\rm c}(Y,{\cal W}({\cal O}_{Y,K_0}))$, 
we have the following exact sequence 
\begin{align*} 
&H^i_{\rm c}(Y,{\cal W}({\cal O}_{Y,K_0}))\lo 
H^i(Z,{\cal W}({\cal O}_{Z,K_0}))\lo 
H^i(Z,{\cal W}({\cal O}_Z/{\cal I})_{K_0})\tag{2.0.4}\label{ali:oyk}\\
&\lo H^i_{\rm c}(Y,{\cal W}({\cal O}_{Y,K_0}))\lo \cdots . 
\end{align*} 
By replacing $Z$ by the closure of $Y$ in $Z$, we see that 
\begin{align*} 
H^i_{\rm c}(Y,{\cal W}({\cal O}_{Y,K_0}))=0
\tag{2.0.5}\label{ali:oydk}
\end{align*} 
if $i>d$. 
Then they have proved that 
there exists the following contravariantly functorial isomorphism 
\begin{align*} 
H^i_{\rm rig,c}(Y/K_0)_{[0,1)}\os{\sim}{\lo} 
H^i_{\rm c}(Y,{\cal W}({\cal O}_{Y,K_0}))
\tag{2.0.6}\label{ali:tk0} 
\end{align*} 
(\cite[Theorem (1.1)]{bbe}). 
\par 
Now let us come back to the case $\kap={\mab F}_q$. 
Since 
\begin{align*} 
H^i_{\rm rig,c}(Y/K_0({\mab F}_q))=\sum_{j=0}^{d-1}
H^i_{\rm rig,c}(Y/K_0({\mab F}_q))_{[j,j+1)}\oplus
H^i_{\rm rig,c}(Y/K_0({\mab F}_q))_{[d]},
\end{align*}  
\begin{align*} 
\# Y({\mab F}_{q})&=\sum_{i=0}^{2d}(-1)^i
\sum_{j=0}^{d-1}{\rm Tr}(F^*_q\vert 
H^i_{\rm rig,c}(Y/K_0({\mab F}_q))_{[j,j+1)})
+\sum_{i=d}^{2d}(-1)^i{\rm Tr}(F^*_q\vert 
H^i_{\rm rig,c}(Y/K_0({\mab F}_q))_{[d]}). \tag{2.0.7}\label{ali:trslfh} 
\end{align*} 
Hence we have the following congruence by (\ref{ali:oydk}) and 
(\ref{ali:tk0}): 
\begin{align*} 
\# Y({\mab F}_{q})\equiv 
\sum_{i=0}^{d}(-1)^i{\rm Tr}
(F^*_q\vert H^i_{\rm c}(Y,{\cal W}({\cal O}_{Y,K_0})))~\mod~q 
\tag{2.0.8}\label{ali:yftrys}
\end{align*} 
in $\ol{{\cal W}({\mab F}_q)}$. 

\begin{rema}
In \cite[(1.4)]{bbe}, the following zeta function 
\begin{align*} 
Z^{\cal W}(Y/{\mab F}_q,t):=\prod_{i=0}^{d}
{\rm det}(1-tF^*_q\vert 
H^i_{\rm c}(Y,{\cal W}({\cal O}_Y))_{K_0})^{(-1)^{i+1}}
\end{align*} 
which is equal to the zeta function 
$$Z^{<1}(Y/{\mab F}_q,t)
:=\prod_{{\rm ord}_q(\al_{ij})<1}(1-\al_{ij}t)^{(-1)^{i+1}}$$
has been considered. In this article we do not need this zeta function. 
We do not need Ax's theorem in \cite{ax}
(see \cite[Proposition 6.3]{bbe}) either.  
\end{rema}

\section{Proofs of (\ref{theo:vc}) and (\ref{coro:npf})}\label{sec:psso}
It is well-known that the analogue of 
Kodaira's vanishing theorem for projective smooth schemes  
over a field of characteristic $0$ 
(\cite{ko}) do not hold in characteristic $p>0$ in general (\cite{raykv}). 
However, in \cite{ny},  
we have proved the Kodaira vanishing theorem in characteristic $p>0$ 
under the assumption of the finiteness of the 
Yobuko height. 
To state this theorem precisely, 
we recall the definition of the vertical type for a relative log scheme. 
\par 
\par 
For a commutative monoid $P$ with unit element, 
an ideal is, by definition, a subset $I$ of $P$ such that $PI\subset I$. 
An ideal ${\mathfrak p}$ of $P$ is called a prime ideal if 
$P\setminus {\mathfrak p}$ is a submonoid of $P$ (\cite[(5.1)]{klog2}). 
For a prime ideal ${\mathfrak p}$ of $P$, the height ${\rm ht}({\mathfrak p})$ 
is the maximal length of sequence's ${\mathfrak p}\supsetneq  
{\mathfrak p}_1\supsetneq \cdots \supsetneq {\mathfrak p}_r$ of prime ideals of $P$. 
Let $h\col Q\lo P$ be a morphism of monoids.  
A prime ideal ${\mathfrak p}$ of $P$ is said to be horizontal with respect to $h$ 
if $h(Q)\subset P\setminus {\mathfrak p}$ (\cite[(2.4)]{tsp}). 
\par 
Let $Y\lo Z$ be a morphism of fs(=fine and saturated) log schemes.  
Let $h\col Q\lo P$ be a local chart of $g$ such that $P$ and $Q$ are saturated. 
Set 
$$I:=\{a\in P~\vert~a\in {\mathfrak p}~~\text{for any 
horizontal prime ideal of $P$ of height 1 with respect to}~h\}.$$  
Let ${\cal I}_{Y/Z}$ be the ideal sheaf of $M_Y$ generated by ${\rm Im}(I\lo M_Y)$. 
In  \cite[(2.6)]{tsp} 
Tsuji has proved that ${\cal I}_{Y/Z}$ is independent of the choice of the local chart $h$. 
Let ${\cal I}_{Y/Z}{\cal O}_Y$ be the ideal sheaf of ${\cal O}_Y$ 
generated by the image of ${\cal I}_{Y/Z}$. 

\begin{defi}\label{defi:vt} 
We say that $Y/Z$ is {\it of vertical type} 
if ${\cal I}_{Y/Z}{\cal O}_Y={\cal O}_Y$.  
\end{defi}

In \cite[(1.9)]{ny} we have proved the following theorem: 

\begin{theo}[{\bf Log Kodaira Vanishing theorem}]\label{theo:stk} 
Let $Y\lo s$ be a projective log smooth morphism of Cartier type of 
fs log schemes.  
Assume that $\os{\circ}{Y}$ is of pure dimension $d$.  
Let ${\cal L}$ be an ample invertible sheaf on $\os{\circ}{Y}$. 
 Assume that $h^F(\os{\circ}{Y})<\infty$. 
Then $H^i(Y,{\cal I}_{Y/s}\om_{Y/s}\otimes_{{\cal O}_Y}{\cal L})=0$ for $i>0$. 
In particular, if $Y/s$ is of vertical type, then  
$H^i(Y,\om_{Y/s}\otimes_{{\cal O}_Y}{\cal L})=0$ for $i>0$. 
\end{theo}

\par 
Now let us prove (\ref{theo:vc}) and (\ref{coro:npf}) quickly. 
Let the notations be as in (\ref{coro:npf}).  
Since $\om_{X/s}^{-1}$ is ample, 
$H^i(X,{\cal O}_X)=0$ for $i>0$ by (\ref{theo:stk}). 
Hence, by the following exact sequence 
\begin{align*}
0\lo {\cal W}_{n-1}({\cal O}_X)\os{V}{\lo}  {\cal W}_{n}({\cal O}_X)\lo 
{\cal O}_X\lo 0, 
\tag{3.2.1}\label{ali:yexyk} 
\end{align*}
$H^i(X,{\cal W}_n({\cal O}_X))=0$ for $i>0$ and $n>0$. 
Hence 
\begin{align*} 
H^i(X,{\cal W}({\cal O}_X))=
(\vpl_nH^i(X,{\cal W}_n({\cal O}_X)))=0.
\tag{3.2.2}\label{ali:ywyk} 
\end{align*} 
Thus we have proved (\ref{theo:vc}).
\par 
Next let us prove (\ref{coro:npf}).
It suffices to prove (\ref{coro:npf}) for the case $k=1$ 
by considering the base change 
$X\otimes_{{\mab F}_q}{\mab F}_{q^k}$. 
Because $H^0(X,{\cal W}({\cal O}_X))={\cal W}({\mab F}_q)$ and 
$F^*_q={\rm id}$ on $H^0(X,{\cal W}({\cal O}_X))$, 
we obtain the following by (\ref{ali:yftrys}):  
\begin{align*} 
\# \os{\circ}{X}({\mab F}_{q})\equiv 1~\mod~q 
\tag{3.2.3}\label{ali:yftrays}
\end{align*} 
in $\ol{{\cal W}({\mab F}_q)}$. 
This shows  (\ref{coro:npf}).

\begin{rema}\label{rema:fsa}
(1) If $X$ is a Fano variety over ${\mab Q}$, then  
the reduction ${\cal X}\mod p$ of a flat model ${\cal X}$ over 
${\mab Z}$ of $X$ for $p\gg 0$ 
is a Fano variety and $F$-split (\cite[Exercise 1.6. E5]{bm}). 
In particular, $h^F({\cal X}~{\rm mod}~p)<\infty$ for $p\gg 0$.  
\par 
(2) As pointed out in \cite[p.~58]{bm}, a Fano variety 
$X$ is not necessarily $F$-split. 
\par 
The Kodaira vanishing theorem does not hold for 
certain Fano varieties (\cite{lr}, \cite{hl}, \cite{to}). 
By (\ref{theo:stk}) we see that the 
Yobuko heights 
of them are infinity. 
\par 
(3) Let $X/s$ be an SNCL Fano scheme of pure dimension $d$. 
Then any irreducible component of 
$\os{\circ}{X}_i$ of $\os{\circ}{X}$ is Fano. 
Indeed, since $\om^{-1}_{X/s}$ is ample, 
$\om^{-1}_{X/s}\otimes_{{\cal O}_X}{\cal O}_{X_i}=
\om^{-1}_{\os{\circ}{X}_i/\kap}(-\sum_{j}\log D_j)$ is also ample. 
Here $\{D_j\}_j$ is the set of the double varieties in $\os{\circ}{X}_i$. 
Hence $-K_{\os{\circ}{X}_i}-\sum_j D_j$ is ample. 
Consequently $-K_{\os{\circ}{X}_i}$ is ample. 
\end{rema}

\begin{rema}\label{rema:ex}
Let $X/{\mab F}_q$ be a separated scheme of finite type. 
Assume that $X$ is geometrically connected.  
By the argument in this section, it is obvious that, 
if $H^i(X,{\cal O}_X)=0$ $(\forall i>0)$, then 
the congruence (\ref{ali:xfq}) holds for $X/{\mab F}_q$.  
In particular, if $d=2$, if $X/{\mab F}_q$ is smooth, 
if $H^1(X,{\cal O}_X)=0$ and if $H^0(X,\Om^2_{X/s})=0$, 
then the congruence (\ref{ali:xfq}) holds for $X/{\mab F}_q$.  
Such an example can be given by 
a proper smooth Godeaux surface. 
\par 
Other examples are given by proper smooth unirational threefolds  
because $H^i(X,{\cal O}_X)=0$ $(\forall i>0)$ 
by \cite[Introduction, (2.5)]{nyn}. 
\par 
Let $X/s$ be 
an SNCL(=simple normal crossing log) 
classical Enriques surface $X/s$ for $p\not=2$, i.e.,  
$(\Om^{2}_{X/s})^{\otimes 2}$ is trivial and the corresponding \'{e}tale 
covering $X'$ to $\Om^{2}_{X/s}$ is
an SNCL $K3$ surface 
(In \cite[(7.1)]{nlk3} 
we have proved that $H^i(X,{\cal O}_X)=0$ for $i>0$.). 
Hence the congruence (\ref{ali:xfq}) 
also holds for $X/s_{{\mab F}_q}$.   
See (\ref{coro:en}) below for the zeta function of this example. 
By the formulas for the zeta function 
((\ref{eqn:kfend}), (\ref{eqn:kenad})), 
we can easily verify that 
$\# \os{\circ}{X}({\mab F}_q)$ indeed satisfies 
the congruence (\ref{ali:xfq}). 
\par 
More generally, if $H^i(X,{\cal W}({\cal O}_X))_{K_0}=0$ 
$(i>0)$, then the congruence (\ref{ali:xfq}) holds for $X/{\mab F}_q$
by the proof of (\ref{theo:vc}). 
By the main theorem of \cite{beru}, 
one obtains such examples which are special fibers of regular proper flat schemes 
over discrete valuation rings of mixed characteristics 
whose generic fibers are geometrically connected and 
of Hodge type $\geq 1$ in positive degrees. 
See also \cite{e} 
for a generalization of the main theorem in \cite{beru}. 
\end{rema} 

\begin{exem}\label{exem:nNp}
Let $n$ and $N$ be positive integers.   
Set ${\cal X}_1:={\mab P}^N_{{\cal W}(\kap)}$. 
Blow up ${\cal X}_1$ along 
an $\kap$-rational hyperplane of ${\mab P}^N_{\kap}$
and let ${\cal X}_2$ be the resulting scheme. 
Let $\os{\circ}{X}_1$ and $\os{\circ}{X}_n$ be 
the irreducible components of 
the special fiber ${\cal X}_2$. 
Blow up ${\cal X}_2$ again along 
$\os{\circ}{X}_1\cap \os{\circ}{X}_n$ 
and let ${\cal X}_3$ be the resulting scheme. 
Let $\os{\circ}{X}_1$, $\os{\circ}{X}_n$ and $\os{\circ}{X}_{n-1}$ be 
the irreducible components of 
the special fiber ${\cal X}_3$. 
Blow up ${\cal X}_3$ again along 
$\os{\circ}{X}_1\cap \os{\circ}{X}_{n-1}$.  
Continuing this process $(n-1)$-times, we have a projective 
semistable family ${\cal X}_n$ over ${\rm Spec}({\cal W}(\kap))$. 
Let $\os{\circ}{X}_i$ $(1\leq i\leq n)$ be the 
the irreducible components of the special fiber ${\cal X}_n$. 
Let $X$ be the log special fiber of ${\cal X}_n$. 
Then $X$ is a projective SNCL scheme over $s$. 
Let $\os{\circ}{X}{}^{(i)}$ $(i=0,1)$ be the disjoint union of 
$(i+1)$-fold intersections of the irreducible components 
of $\os{\circ}{X}$. 
Then $\os{\circ}{X}{}^{(0)}={\mab P}^N_{\kap}\coprod 
\underset{n-1~{\rm times}}
{\underbrace{({\mab P}^{N-1}_{\kap}\times_{\kap}{\mab P}^1_{\kap})\coprod
\cdots \coprod({\mab P}^{N-1}_{\kap}\times_{\kap}{\mab P}^1_{\kap})}}$ 
and 
$\os{\circ}{X}{}^{(1)}=\underset{n-1~{\rm times}}
{\underbrace{{\mab P}^{N-1}_{\kap}\coprod
\cdots \coprod{\mab P}^{N-1}_{\kap}}}$. 
Using the following spectral sequence 
\begin{align*} 
E_1^{ij}=H^j(X^{(i)},{\cal O}_{X^{(i)}})\Lo 
H^{i+j}(X,{\cal O}_X)
\tag{3.5.1}\label{ali:exi} 
\end{align*} 
and noting that the dual graph of $\os{\circ}{X}$ is a segment, 
we see that $H^i(X,{\cal O}_X)=0$ $(i>0)$.  
If $s=s_{{\mab F}_q}$, then it is easy to check that 
\begin{align*} 
\os{\circ}{X}({\mab F}_q)&=\dfrac{q^{N+1}-1}{q-1}
+(n-1)\dfrac{q^{N}-1}{q-1}\dfrac{q^2-1}{q-1}
-(n-1)\dfrac{q^{N}-1}{q-1}\\
&=\dfrac{q^{N+1}-1}{q-1}+
(n-1)q\dfrac{q^N-1}{q-1}.
\end{align*}  
In particular, $\#\os{\circ}{X}({\mab F}_q)\equiv 1~{\rm mod}~q$.
\par 
The restriction of $\om_{X/s}$ to $\os{\circ}{X}_i$ 
is isomorphic to ${\cal O}_{\os{\circ}{X}_i}(-N)$ for $i=0$, $N$ 
and ${\cal O}_{\os{\circ}{X}_i}(-(N-1))$ for $0< i<N$. 
Hence $\om_{X/s}^{-1}$ is ample if $N\geq 2$. 
Since each $\os{\circ}{X}_i$ is $F$-split 
(the $F$-splitting is given by the ``$p^{-1}$-th power'' of the canonical coordinate 
of $\os{\circ}{X}_i$ (see \cite[(1.1.5)]{bm})
and because 
we have the following exact sequence 
\begin{align*} 
0\lo {\cal O}_X\lo \bigoplus_{i=1}^{N}{\cal O}_{X_i}\lo 
\bigoplus_{i=1}^{N-1}{\cal O}_{X_{i}\cap X_{i+1}},  
\end{align*}
$\os{\circ}{X}$ is $F$-split. 
 
\end{exem}

\section{Proof of (\ref{theo:xfh})}\label{sec:cy}
Let the notations be as in (\ref{theo:xfh}). 
In this section we prove (\ref{theo:xfh}). 
We may assume that $k=1$. 
\par 
Since $H^{d-1}(Y,{\cal O}_Y)=0$,   
we see that 
\begin{align*} 
H^{d-1}(Y,{\cal W}({\cal O}_Y))=
\vpl_nH^i(Y,{\cal W}_n({\cal O}_Y))=0
\tag{4.0.1}\label{ali:yclk} 
\end{align*} 
by the same proof as that of (\ref{theo:vc}). 
Set $\ol{Y}:=Y\otimes_{{\mab F}_q}\ol{\mab F}_{q}$ 
and $e:=\log_pq$. 
By \cite[II (4.3)]{am} the Dieudonn\'{e} module 
$M:=D(\Phi_{Y/\kap})$ of $\Phi_{Y/\kap}$ is equal to 
$H^d(Y,{\cal W}({\cal O}_Y))$.  
Let $h$ be the height of $\Phi_{Y/\kap}$. 
Hence 
$\Phi_{Y/\kap}$ is a 
commutative formal Lie group over $\kap$ of dimension 1 
and  the Dieudonn\'{e} module $M$ 
is a free ${\cal W}$-module of rank $h$ if 
$h< \infty$ (\cite[V (28.3.10)]{ha}). 
Let $F\col M\lo M$ be 
the operator ``$F$'' on the Dieudonn\'{e} module $M$. 
By abuse of notation, we denote 
the induced morphism 
$M_{K_0}\lo M_{K_0}$ by $F$. 
By (\ref{ali:yftrys}) we have the following congruence 
\begin{align*} 
\# Y({\mab F}_{q})\equiv 1+{\rm Tr}(F^e\vert M_{K_0})
~\mod~q 
\tag{4.0.2}\label{ali:yftrcys}
\end{align*} 
in $\ol{{\cal W}({\mab F}_q)}$. 
Set $m:={\rm ord}_p({\rm Tr}(F^e\vert M_{K_0}))$.  If 
$m\leq e={\rm ord}_p(q)$, then  
we obtain the following congruence by (\ref{ali:yftrcys}): 
\begin{align*} 
\# Y({\mab F}_{q})\equiv 1
~\mod~p^{\lceil m \rceil} 
\tag{4.0.3}\label{ali:yftrcemys}
\end{align*} 
in ${\mab Z}$. 

\par 
First we give the proof of (\ref{theo:xfh}) (1).
\medskip 
\parno  
{\bf Proof of (\ref{theo:xfh}) (1).}
\medskip 
\par 
Assume that $h=\infty$. 
Then $D(\Phi_{\ol{Y}/\ol{\mab F}_q})$ is ${\cal W}(\ol{\mab F}_q)$-torsion. 
By \cite[II (4.3)]{am}, 
$H^d(\ol{Y},{\cal W}({\cal O}_{\ol{Y}}))_{K_0}=
D(\Phi_{\ol{Y}/\ol{\mab F}_q})_{K_0}=0$. 
By \cite[I (1.9.2)]{idw}, ${\cal W}({\cal O}_{\ol{Y}})
={\cal W}({\cal O}_Y)\otimes_{{\cal W}({\mab F}_q)}{\cal W}(\ol{\mab F}_q)$. 
Since $\os{\circ}{Y}$ is separated, 
we obtain the following equality 
$H^d(Y,{\cal W}({\cal O}_{\ol{Y}}))=
H^d(Y,{\cal W}({\cal O}_Y))\otimes_{{\cal W}({\mab F}_q)}{\cal W}(\ol{\mab F}_q)$ by using \v{C}ech cohomologies. 
Hence 
$$H^d(Y,{\cal W}({\cal O}_Y))_{K_0({\mab F}_q)}=0.$$ 
(To obtain this vanishing, one may use the fact that 
the Dieudonn\'{e} module commutes with base change (cf.~the description of 
$D(\Phi_{Y/{\mab F}_q})$ in \cite[p.~309]{mub}.))
By (\ref{ali:yftrys})
this means the congruence (\ref{eqn:kfdintd}). 
\par 
Now assume that  $h<\infty$. 
Next we give the proof  (\ref{theo:xfh}) (2).
\medskip 
\parno  
{\bf Proof of (\ref{theo:xfh}) (2).}
\medskip
\par
\par 
Let us recall the following well-known observation  
(\cite[Exercise 6.13]{li}): 

\begin{prop}\label{prop:eobs}
Let $G$ be 
a commutative formal Lie group of dimension $1$ 
over a perfect field $\kap$ of characteristic $p>0$. 
Assume that the height $h$ of $G$ is finite. 
Then the slopes of the Dieudonn\'{e} module of $D(G)$ 
is $1-h^{-1}$. 
\end{prop}
\begin{proof} 
Let $D(\kap)$ be the Cartier-Dieudonn\'{e} algebra over $\kap$. 
We may assume that $\kap$ is algebraically closed. 
In this case, the height is the only invariant which determines the 
isomorphism class of a 1-dimensional commutative formal group law 
over $\kap$ 
(\cite[(19.4.1)]{ha}). 
Hence  
$D(G)\simeq D(\kap)/D(\kap)(F-V^{h-1})$ (\cite[p.~266]{vgk}). 
Express $F(1,V, \cdots, V^{h-1})=(1,V, \cdots, V^{h-1})A$, 
where $A\in M_h({\cal W})$ (as if $F$ were ${\cal W}$-linear). 
Then ${\rm det}(tI-A)=t^h-p^{h-1}$. 
Hence the slopes of $D(G)$ is ${\rm ord}_p((p^{h-1})^{h^{-1}})=1-h^{-1}$.
\end{proof} 

\par 
By (\ref{prop:eobs}) and (\ref{ali:yftrcemys}),  
we obtain the following congruence 
\begin{align*} 
\# Y({\mab F}_{q})\equiv 1
~\mod~p^{\lceil e(1-h^{-1})\rceil} 
\tag{4.1.1}\label{ali:yftzys}
\end{align*} 
in ${\mab Z}$.

Lastly we give the proof of (\ref{theo:xfh}) (3) in the following.
\medskip 
\parno  
{\bf Proof of (\ref{theo:xfh}) (3).}
\medskip 
\par 
Let $\kap$ be a perfect field of characteristic $p>0$. 
Let $Y$ be a proper scheme over $\kap$ of pure dimension $d\geq 1$.
$($We do not assume that $Y$ is smooth over $\kap$.$)$ 
Assume that $H^d(Y,{\cal O}_Y)\simeq \kap$ and 
that $H^{d-1}(Y,{\cal O}_Y)=0$ if $d\geq 2$.  
Then the following morphism 
\begin{align*} 
H^d(Y,{\cal W}({\cal O}_Y))/p\lo H^d(Y,{\cal O}_Y)
\end{align*} 
is an isomorphism. 
Indeed, this is surjective and 
\begin{align*} 
{\rm dim}_{\kap}(H^d(Y,{\cal W}({\cal O}_Y))/p)=
{\rm dim}_{\kap}(M/p)=
1={\rm dim}_{\kap}H^d(Y,{\cal O}_Y).
\end{align*}  
Since $h=1$, $F$ on 
$H^d(Y,{\cal W}({\cal O}_Y))\otimes_{{\cal W}(\kap)}{\cal W}(\ol{\kap})$ 
is an isomorphism, Hence
$F\col H^d(Y,{\cal O}_Y)\lo H^d(Y,{\cal O}_Y)$ is an isomorphism. 
Hence $\# Y({\mab F}_{q})\equiv 1+\al~{\rm mod}~q$ 
for a unit $\al \in {\cal W}({\mab F}_q)^*$.  
Now  (\ref{theo:xfh}) (3) follows.  
\parno 

\begin{rema}\label{rema:fe}
(1) If $H^i(Y,{\cal O}_Y)=0$ for $0<i<d-1$ 
(this is stronger than (c) in the Introduction),  
then 
(\ref{theo:xfh}) (3) also follows from Fulton's trace formula (\cite{ftf}): 
\begin{align*} 
\# Y({\mab F}_{q})~{\rm mod}~p\equiv 
\sum_{i=0}^{d}(-1)^i{\rm Tr}(F^*_q\vert H^i(Y,{\cal O}_Y))\in {\mab F}_q 
\end{align*} 
(cf.~\cite[Proposition 5.6]{bth}).  
\par 
(2) 
Let $X/s$ be a log Calabi-Yau scheme. 
In \cite[(10.1)]{ny} we have proved a fundamental equality 
$h^F(X/\kap)=h(X/\kap)$. 
Hence $X$ is quasi-$F$-split (resp.~$F$-split) if and only if 
$h(X/\kap)<\infty$ (resp.~$h(X/\kap)=1$). 
\end{rema}

Though the following corollary immediately follows from \cite[(1.6)]{bbe}, 
we state it for the convenience of our remembrance. 

\begin{coro}\label{coro:ht} 
Let  $Y$ be as in {\rm (\ref{theo:xfh})}. 
Let $f\col Z_1\lo Z_2$ be a morphism of proper schemes over ${\mab F}_q$. 
Assume that $Z_1$ or $Z_2$ is isomorphic to $Y$ over ${\mab F}_q$. 
Assume that $\Phi_{Z_i/{\mab F}_q}$ $(i=1,2)$ is representable. 
If the pull-back $f^*\col H^i(Z_2,{\cal O}_{Z_2})\lo H^i(Z_1,{\cal O}_{Z_1})$ 
is an isomorphism, then the natural morphism 
$\Phi_{Z_2/{\mab F}_q}\lo \Phi_{Z_1/{\mab F}_q}$ is an isomorphism. 
In particular, $h(Z_1/{\mab F}_q)=h(Z_2/{\mab F}_q)$ 
and {\rm (\ref{theo:xfh})} for $\# Z_i({\mab F}_q)$ holds. 
\end{coro} 
\begin{proof} 
By the assumption, we have an isomorphism 
$f^*\col H^i(Z_2,{\cal W}({\cal O}_{Z_2}))\os{\sim}{\lo} 
H^i(Z_1,{\cal W}({\cal O}_{Z_1}))$. 
Hence the natural morphism 
$D(\Phi_{Z_2/{\mab F}_q})\lo D(\Phi_{Z_1/{\mab F}_q})$ 
is an isomorphism. 
By Cartier theory, the natural morphism 
$\Phi_{Z_2/{\mab F}_q}\lo \Phi_{Z_1/{\mab F}_q}$ is an isomorphism. 
This implies that $h(Z_2/{\mab F}_q)=h(Z_2/{\mab F}_q)$ 
and {\rm (\ref{theo:xfh})} for $\# Z_i({\mab F}_q)$ holds. 
\end{proof} 

The following corollary immediately follows from the proof of 
\cite[(6.12)]{bbe}. 

\begin{coro}\label{coro:eo}
Let $Y$ be as in {\rm (\ref{theo:xfh})}. 
Let $G$ be a finite group acting on $Y/{\mab F}_q$ such that 
each orbit of $G$ is contained in an affine open subscheme of $Y$. 
If $\# G$ is prime to $p$ and the induced action on $H^d(Y,{\cal O}_Y)$ 
is trivial, then $h((Y/G)/{\mab F}_q)=h(Y/{\mab F}_q)$ and {\rm (\ref{theo:xfh})} for 
$\# (Y/G)({\mab F}_q)$ holds.  
\end{coro} 


\begin{exem}\label{exem:gnt} 
We give examples of trivial logarithmic cases. 
\par 
(1) Let $E/{\mab F}_p$ be an elliptic curve. 
It is very well-known that 
$E/{\mab F}_p$ is supersingular if and only if 
$\# E({\mab F}_p)=p+1$ if $p\geq 5$ (\cite[V Exercises 5.9]{sil}). 
As observed in \cite[Example 5.11]{bth}, 
this also follows from the purity of 
the weight for an elliptic curve over ${\mab F}_p$: 
$\vert \# E({\mab F}_p)-(p+1)\vert \leq 2\sqrt{p}$ 
and Fulton's trace formula.   
In fact, we can say more in (\ref{prop:cg}) below. 
\par 
(2) 
Let $d\geq 3$ be a positive integer such that $d\not\equiv 0~{\rm mod}~p$. 
Consider a smooth Calabi-Yau variety 
${\cal X}/{\cal W}({\mab F}_q)$ 
in ${\mab P}^{d-1}_{{\cal W}({\mab F}_q)}$
defined by the following equation: 
\begin{align*} 
a_0T^d_0+\cdots +a_{d-1}T^d_{d-1}=0
\quad (a_0, \ldots,a_{d-1}\in {\cal W}({\mab F}_q)^*). 
\end{align*} 
Set $a:=a_0\cdots a_{d-1}\in {\cal W}({\mab F}_q)$. 
Let $X/{\mab F}_q$ be the reduction mod~$p$ of 
${\cal X}/{\cal W}({\mab F}_q)$. 
By \cite[Theorem 1]{stfg} (see also [loc.~cit., Example 4.13]), 
the logarithm $l(t)$ of $\Phi_{{\cal X}/{\cal W}({\mab F}_q)}$ is given by the following formula:
\begin{align*} 
l(t)=\sum_{m=0}^{\infty}a^m\dfrac{(md)!}{(m!)^d}\dfrac{t^{md+1}}{md+1}. 
\end{align*} 
\par 
(a) 
If $p\equiv 1~{\rm mod}~d$, 
then 
$$pl(t)=pt+p\sum_{i=2}^{p-1}c_it^i
+({\rm unit})t^p+({\rm higher~terms~than~}t^p)$$ 
for some $c_i\in {\cal W}({\mab F}_q)$ 
in ${\cal W}({\mab F}_q)[[t]]$. 
Hence $l^{-1}(pl(t))~{\rm mod}~p\equiv t^p+\cdots$ and 
the height of $\Phi_{X/{\mab F}_q}$ is equal to $1$. 
\par 
(b) If  $p\not\equiv 1~{\rm mod}~d$, then $pl(t)\in p{\cal W}({\mab F}_q)[[t]]$. 
Hence the height of $\Phi_{X/{\mab F}_q}$ is equal to $\infty$. 
\par 
(a) and (b) above are much easier and much more direct proofs 
of \cite[Theorem 5.1]{vgkht}.
\par 
(3) Especially consider the case $N=3$ in (2) and let $X/{\mab F}_p$ 
be a closed subscheme of ${\mab P}^N_{{\mab F}_p}$
defined by the following equation: 
\begin{align*} 
T^4_0+T^4_1+T^4_2+T^4_3=0. 
\end{align*} 
\par 
(a) If $p=3$, then $\# X({\mab F}_3)\equiv 1~{\rm mod}~3$ by (\ref{theo:xfh}) (1). 
In fact, it is easy to see that 
$\# X({\mab F}_3)=4=1+3^2-3\times 2$. 
(This $X$ and $X$ in (c) are Tate's examples in \cite{ta} 
of a supersingular $K3$-surface 
(in the sense of T.~Shioda)  
over ${\mab F}_3$ and ${\mab F}_7$), respectively.) 
\par 
(b) If $p=5$, then $\# X({\mab F}_5)\not\equiv 1~{\rm mod}~5$
by (\ref{theo:xfh}) (3). 
In fact, it is easy to see that $\# X({\mab F}_5)=0$. 
More generally, for a power $q$ of a prime number $p$, 
let $X_q$ be a closed subscheme 
of ${\mab P}^{q-1}_{{\mab F}_q}$ defined by the following equation: 
\begin{align*} 
a_0T^{q-1}_0+\cdots +a_{q-2}T^{q-1}_{q-2}=0
\quad (a_0, \ldots,a_{q-2} \in {\mab F}_q^*,~(a_0, \ldots,a_{q-2})\not=(0,\ldots,0)),   
\end{align*} 
where $a_0, \ldots,a_{q-2}$ satisfying the following condition: 
for any nonempty set $I$ of $\{0, \ldots, q-2\}$,  
$\sum_{j\in I}a_j\not=0$ in ${\mab F}_q$. 
Then $\# X_q({\mab F}_q)=0$. 
\par 
(c) If $p=7$, then $\# X({\mab F}_7)\equiv 1~{\rm mod}~7$
by (\ref{theo:xfh}) (1). 
In fact, one can check that 
$\# X({\mab F}_7)=64=1+7^2+7\times 2$.  
In general, 
if $\Phi(X/{\mab F}_q)$ is supersingular, 
then  $\# X({\mab F}_q)=1+q^2+q\al$ 
for some $\vert \al \vert \leq 22$ 
by the purity of the weight and by $b_2(\ol{X})= 22$. 
Here $b_2(\ol{X}/\ol{\mab F}_q)$ 
is the second Betti number of $\ol{X}/\ol{\mab F}_q$. 
(We do not know an example of the big $\vert a \vert$.)  
\par 
(4) See \cite[(4.8)]{yy} for explicit examples of $X/{\mab F}_q$'s  
such that $h(\Phi_{X/{\mab F}_q})=2$. 
See also \cite[\S6]{vgkht}. 
\end{exem} 



\begin{exem} 
(1) Let $n$ be a positive integer. Let $X$ be an $n$-gon over ${\mab F}_q$.   
Then, by \cite[(6.7) (1)]{nlfc}, 
$X$ is $F$-split. In particular, $h^F(X)=h(X/{\mab F}_q)=1$. 
Then, by (\ref{theo:xfh}) (3), 
$\# X({\mab F}_q)\not\equiv 1~{\rm mod}~p$. 
In fact, it is easy to see that 
$\# X({\mab F}_q)=n(q+1)-n=q$. 
Compare this example with the example in (\ref{exem:nNp}). 
\par 
(2) Let $\kap$ be a perfect field of characteristic $p>0$. 
Let $X$ be an SNCL(=simple normal crossing log) $K3$-surface over $\kap$, 
that is, an SNCL Calabi-Yau variety of dimension $2$ (\cite{nlk3}). 
In \cite[(6.7) (2)]{nlfc} we have proved the following: 
\par 
(a) If $\os{\circ}{X}$ is of Type II {\rm (\cite[\S3]{nlk3})}, 
then $X$ is $F$-split if and only if the isomorphic double elliptic curve is ordinary. 
In this case, $h(X/\kap)=1$. If this is not the case,  $h(X/\kap)=2$. 
\par 
(b) If $\os{\circ}{X}$ is of Type III {\rm ([loc.~cit.])}, 
then $X$ is $F$-split and $h(X/\kap)=1$.
\par 
See (\ref{coro:f}) below for the zeta function of these examples. 
By the formulas for the zeta function ((\ref{eqn:kfmdad}) and (\ref{eqn:kfdqad})), 
we can easily verify that $\# \os{\circ}{X}({\mab F}_q)$ indeed satisfies 
the congruences (\ref{eqn:kfdbntd}) and (\ref{eqn:kfd2td}). 
\end{exem} 

\begin{rema} 
(1) Let $X/{\mab F}_q$ and $X^*/{\mab F}_q$ be a strong mirror Calabi-Yau pair  
in the sense of Wan (\cite{wama}), 
whose strict definition has not been given. 
Then he conjectures that 
$\# X({\mab F}_q)\equiv \# X^*({\mab F}_q)~{\rm mod}~q$ (\cite[(1.3)]{wama}). 
Hence the following question seems natural: 
does the equality 
$h(\Phi_{X/{\mab F}_q})=h(\Phi_{X^*/{\mab F}_q})$ hold? 
If his conjecture is true, 
only one of $h(\Phi_{X/{\mab F}_q})$ and $h(\Phi_{X^*/{\mab F}_q})$ 
cannot be 1 by (\ref{theo:xfh}). 
This is compatible with Wan's generically ordinary conjecture 
in [loc.~cit., (8.3)]. 
\par 
(2)  If $X$ satisfies the conditions (a), (c) and (d) in the Introduction 
and if $X$ is a special fiber of a regular proper flat scheme 
over a discrete valuation ring of mixed characteristics 
whose generic fibers are geometrically connected 
and of Hodge type $\geq 1$ in degrees in $[1,d-2]$, 
then we see that $X$ satisfies the condition (b) by \cite{beru}. 
\end{rema}

We conclude this section by generalizing (\ref{exem:gnt}) (1) 
by using (\ref{theo:xfh}) and 
raise an important question: 

\begin{prop}\label{prop:cg}
Let $C$ be a proper smooth curve over ${\mab F}_{q}$ 
such that $H^0(C,{\cal O}_C)\simeq {\mab F}_{q}\simeq 
H^1(C,{\cal O}_C)$. Recall that $e=\log_pq$. 
Then the following hold$:$
\par
$(1)$ Assume that $e$ is odd and $p\geq 5$.  
Then $h_{C/{\mab F}_q}=2$ if and only if 
$\#C({\mab F}_q)=1+q$. 
\par
$(2)$ Assume that $e$ is odd and $p= 3$ or $2$. 
Then $h_{C/{\mab F}_q}=2$ if and only if 
$\#C({\mab F}_q)=1+q$ or $1+ q\pm p^{\frac{e+1}{2}}$. 
\par
$(3)$ Assume that $e$ is even. 
Then $h_{C/{\mab F}_q}=2$ if and only if 
$\#C({\mab F}_q)=1+q+\al p^{\frac{e}{2}}$, where $\al \in {\mab N}$ and $\vert \al \vert \leq 2$. 
\end{prop} 
\begin{proof} 
By the purity of weight, we have the following inequality: 
\begin{align*} 
\vert \#C({\mab F}_{q})- (1+q)\vert \leq 2\sqrt q.
\tag{4.8.1}\label{ali:caqq} 
\end{align*}  
\par 
(1): Assume that $h_{C/{\mab F}_q}=2$. 
By (\ref{eqn:kfd2td}), 
$\# C({\mab F}_q)\equiv 1~{\rm mod}~ p^{\lceil \frac{e}{2}\rceil}
=1~{\rm mod}~ p^{\frac{e+1}{2}}$.  
Hence $\# C({\mab F}_{q})=1+m p^{\frac{e+1}{2}}$ for $m \in {\mab N}$. 
By (\ref{ali:caqq}) we have the following inequality: 
\begin{align*} 
p^{\frac{1}{2}}\vert m - p^{\frac{e-1}{2}}\vert \leq 2.
\tag{4.8.2}\label{ali:cqq} 
\end{align*} 
Since $p\geq 5$, $m = p^{\frac{e-1}{2}}$. Hence 
$\# C({\mab F}_q)= 1+q$. 
\par 
Conversely, assume that $\# C({\mab F}_q)=1+q$. 
Then $C$ can be an elliptic curve over ${\mab F}_q$. 
Hence $h_{C/{\mab F}_q}=1$ or $2$ 
(\cite[IV (7.5)]{sil}). 
By (\ref{eqn:kfdbntd}) and (\ref{eqn:kfd2td}), $h_{C/{\mab F}_q}=2$. 
\par 
(2): Assume that $h_{C/{\mab F}_q}=2$. 
Then, by (\ref{ali:cqq}), 
$m=p^{\frac{e-1}{2}}$ or  $m=\pm 1+p^{\frac{e-1}{2}}$. 
Hence $\# C({\mab F}_q)= 1+q$ or 
$\# C({\mab F}_q)=1+(\pm 1+p^{\frac{e-1}{2}})p^{\frac{e+1}{2}}=1+q\pm p^{\frac{e+1}{2}}$.  
\par 
The proof of the converse implication is the same as that in (1). 
\par 
(3): Assume that $h_{C/{\mab F}_q}=2$. 
By (\ref{eqn:kfd2td}), 
$\# C({\mab F}_q)\equiv 1~{\rm mod}~ p^{\frac{e}{2}}$.  
Hence $\# C({\mab F}_{q})=1+m p^{\frac{e}{2}}$ for $\al \in {\mab N}$. 
By (\ref{ali:caqq}), 
\begin{align*} 
\vert m - p^{\frac{e}{2}}\vert \leq 2.
\tag{4.8.3}\label{ali:cbqq} 
\end{align*} 
Hence, by (\ref{ali:caqq}), 
$m=\al +p^{\frac{e}{2}}$ with $\vert \al \vert \leq 2$. 
Hence $\# C({\mab F}_q)=1+(\al+p^{\frac{e}{2}})p^{\frac{e}{2}}
=1+q+\al p^{\frac{e}{2}}$.  
\par 
The proof of the converse implication is the same as that in (1). 
\end{proof} 


\begin{rema} 
Assume that $e$ is even. 
By Honda-Tate's  theorem for elliptic curves over finite fields
(\ref{theo:pap}) below, 
the case $\vert \al \vert =1$ occurs only when 
$p\not\equiv 1~{\rm mod}~3$; 
the case $\al  =0$ occurs only when 
$p\not\equiv 1~{\rm mod}~4$. 
\end{rema}

\begin{theo}[{\bf Honda-Tate's theorem 
for elliptic curves (\cite[(4.1)]{wat}, \cite[(4.8)]{pap})}]\label{theo:pap} 
For an elliptic curve $E/{\mab F}_q$, set $t_E:=1+q-\# E({\mab F}_q)$. 
Consider the following well-defined injective map$\!:$ 
\begin{align*} 
\{{\rm isogeny}~{\rm classes}~{\rm of}~
{\rm elliptic}~{\rm curves}~E/{\mab F}_q\}\owns E
\lo 
t_E\in \{t\in {\mab Z}~\vert~\vert t \vert \leq 2\sqrt{q}\}. 
\end{align*} 
$($This map is indeed injective by Tate's theorem $(${\rm 
\cite[Main Theorem]{tae}}.$)$
The image of HT consists of the following values$:$
\par 
$(1)$ $t$ is coprime to $p$. 
\par 
$(2)$ $e$ is even and $t=\pm 2 \sqrt{q}$. 
\par 
$(3)$ $e$ is even and $p\not\equiv 1~{\rm mod}~3$ and $t=\pm \sqrt{q}$. 
\par 
$(4)$ $e$ is odd and $p=2$ or $3$ and 
$t=\pm p^{\frac{e+1}{2}}$. 
\par 
$(5)$ $e$ is odd, or $e$ is even 
and $p\not\equiv 1~{\rm mod}~4$ and $t=0$.
\par 
The case $(1)$ arises from ordinary elliptic curves over ${\mab F}_q$.  
The case $(2)$ arises from supersingular elliptic curves over ${\mab F}_q$ 
having all their endomorphisms defined over ${\mab F}_q\!;$ 
the rest cases arises from 
supersingular elliptic curves over ${\mab F}_q$ not having 
all their endomorphisms defined over ${\mab F}_q$.
\end{theo}

\begin{prob}\label{prob:dist}
Let $K$ be an algebraic number field and 
${\cal O}_K$ the integer ring of $K$. 
Let $x$ be a positive real number. 
\par 
(1) 
Consider the following set 
\begin{align*} 
{\cal P}(x)
:=\{{\mathfrak p}\in {\rm Spec}({\cal O}_K)~\vert~
N_{K/{\mab Q}}({\mathfrak p})\leq x ~
{\rm and}~
\log_p(\# ({\cal O}_K/{\mathfrak p}))~{\rm is~even}\}, 
\end{align*} 
where $p={\rm ch}({\cal O}_K/{\mathfrak p})$. 
\par 
Assume that $p\geq 5$. 
Let $E/K$ be an elliptic curve. 
Let $\al$ be an integer such that $\vert \al \vert \leq 2$. 
Consider the following set 
\begin{align*} 
{\cal P}'(x;E/K,\al)
:=\{{\mathfrak p}\in {\cal P}(x)~\vert~&
E~{\rm has~a~good~reduction~}{\cal E}_0~{\rm at}~{\mathfrak p}\\
&{\rm and}~
\#{\cal E}_0({\mab F}_q)=1+q+\al \sqrt{q}\}.  
\end{align*}
Set 
\begin{equation*}
{\cal P}(x;E/K,\al) :=
\begin{cases} 
{\cal P}'(x;E/K,\al)& (\vert \al \vert =2), \\
\{{\mathfrak p}\in {\cal P}(x;E/K,\al)~\vert~p\not\equiv 1~{\rm mod}~3\}
& (\vert \al \vert =1),\\
\{{\mathfrak p}\in {\cal P}(x;E/K,\al)~\vert~p\not\equiv 1~{\rm mod}~4\}
& (\al =0). 
\end{cases} 
\label{eqn:kfptd}
\end{equation*}
Then, what is the function 
\begin{align*} 
x\lom \dfrac{\# {\cal P}(x;E/K,\al)}{\#{\cal P}(x)}
\end{align*} 
when $x\rightarrow \infty$?
(I do not know whether 
$\lim_{x\rightarrow \infty}{\cal P}(x;E/K,\al)=\infty$ for each 
$\al$ such that $\vert \al \vert \leq 2$ for any non-CM elliptic curve over $K$ 
(see \cite[p.~185 Exercise 2.33 (a), (b)]{sila} 
for a CM elliptic curve over ${\mab Q}(\sqrt{-1})$: in this example, 
$\lim_{x\rightarrow \infty}{\cal P}(x;E/K,2)=\infty$, but 
${\cal P}(x;E/K,\al)=0$ for $\al\not=2$ and for any $x$). 
If $[K:{\mab Q}]$ is odd or 
if $K$ has a real embedding, 
then $\lim_{x\rightarrow \infty}
\sum_{\vert \al \vert \leq 2}{\cal P}(x;E/K,\al)=\infty$ by 
Elkies' theorems (\cite[Theorem 2]{elq}, \cite[Theorem]{elr}).) 
\par 
When $p=2$ or $3$, we can give a similar problem to the problem above 
by using (\ref{prop:cg}) (2). 
\par 
(2) Consider the following set 
\begin{align*} 
{\cal P}(x)
:=\{{\mathfrak p}\in {\rm Spec}({\cal O}_K)~\vert~
N_{K/{\mab Q}}({\mathfrak p})\leq x\}.  
\end{align*} 
\par 
Let $S/K$ be a K3 surface. 
Let $\al$ be an integer 
such that $\vert \al \vert \leq 22$. 
Consider the following set 
\begin{align*} 
{\cal P}'(x;S/K,\al)
:=\{{\mathfrak p}\in {\cal P}(x)~\vert~&
S~{\rm has~a~good~reduction~}{\cal S}_0~{\rm at}~{\mathfrak p}\\
&{\rm and}~\#{\cal S}_0({\mab F}_q)=1+q^2+\al q\}.  
\end{align*}
Then, what is the function 
\begin{align*} 
x\lom \dfrac{\# {\cal P}'(x;S/K,\al)}{\#{\cal P}(x)}
\end{align*} 
when $x\rightarrow \infty$?
(I do not know even whether 
$\lim_{x\rightarrow \infty}
\sum_{\vert \al \vert \leq 22}{\cal P}'(x;S/K,\al)=\infty$.) 
\end{prob}

\section{Two kinds of zeta functions of degenerate SNCL schemes 
over the log point of ${\mab F}_q$}\label{sec:tkz}
In this section we give a few examples of two kinds of local zeta functions of 
a separated scheme $Y$ of finite type over ${\mab F}_q$: 
one of them is defined by rational points of $Y$; 
the other is defined by the Kummer \'{e}tale cohomology of $Y$ 
when $Y$ is the underlying scheme of a proper log smooth scheme 
over the log point $s_{{\mab F}_q}$. 
\par 
First we introduce a Grothendieck group which is convenient 
in this section. 
\par 
Let $F$ be a field. Consider a Grothendieck group ${\cal K}(F)$ 
with the following generators and relations: the generators of 
${\cal K}(F)$ are $[(V,\bet)]$'s, where $V$ is a finite-dimensional 
vector space over $F$ and $\bet$ is an endomorphism of $V$ over $F$.
The relations are  as follows: $[(V,\bet)] = [(U,\al)] +[(W,\gamma)]$ 
for a commutative diagram with exact rows 
\begin{equation*} 
\begin{CD}
0 @>>> U @>>> V @>>> W @>>> 0 \\   
@. @V{\al}VV @V{\bet}VV  @V{\gam}VV  \\ 
0 @>>> U @>>> V @>>> W @>>> 0.  
\end{CD}
\end{equation*} 
\par
Let $t$ be a variable. 
Note that ${\rm det}(1-t\bet \vert V)={\rm det}(1-t\al \vert U)
{\rm det}(1-t\gamma \vert W)$. If $V=\{0\}$, we set  
${\rm det}(1-t0 \vert V)=1$ $(1\in F)$. 
We have a natural map 
\begin{equation*} 
{\rm det}(1- t\bul \vert \bul) \col {\cal K}(F) \lo F(t)^*\cap 
(1+tF[[t]])^* 
\tag{5.0.1}\label{eqn:grdet} 
\end{equation*} 
of abelian groups. Here the intersection in the target of (\ref{eqn:grdet}) 
is considered in the ring of Laurent power series 
in one variable with coefficients in $F$. 
Set $Z((V,\al),t)={\rm det}(1- t\al \vert V)$. 
\par 
Let $Y$ be a separated scheme of finite type over ${\mab F}_q$. 
Set 
\begin{align*} 
[(E_p(Y/{\mab F}_q),F^*_q)]:=
\sum_{i=0}^{\infty}(-1)^i[(H^i_{\rm rig,c}(Y/K_0({\mab F}_q)),F^*_q)]
\in {\cal K}(K_0({\mab F}_q)), 
\tag{5.0.2}\label{ali:zet}
\end{align*}   
where $E_p$ means the Euler-characteristic. 
Let  
\begin{align*} 
Z(Y/{\mab F}_q,t):={\rm exp}
\left(\sum_{n=0}^{\inf}\dfrac{\# Y({\mab F}_{q^n})}{n}t^n\right)
\tag{5.0.3}\label{ali:zcget}
\end{align*} 
be the zeta function of $Y/{\mab F}_q$.  
We can reformulate (\ref{ali:trfh}) as the following formula:   
\begin{align*} 
Z(Y/{\mab F}_q,t)=Z([(E_p(Y/{\mab F}_q),F^*_q)])^{-1}.  
\tag{5.0.4}\label{ali:trfrgh} 
\end{align*}

\begin{prop}\label{prop:ykf}
Let $Y$ be a proper SNC $($not necessarily log$)$ 
scheme over ${\mab F}_q$. 
Let $Y^{(i)}$ $(i\in {\mab Z}_{\geq 0})$ be the disjoint union of the 
$(i+1)$-fold intersections of the irreducible components of $Y$. 
Then 
\begin{align*} 
Z(Y/{\mab F}_q,t)=\prod_{i,j\geq 0}
{\rm det}(1-tF^*_q\vert H^j_{\rm rig}(Y^{(i)}/K_0({\mab F}_q)))^{(-1)^{i+j+1}}.
\tag{5.1.1}\label{ali:zytp}
\end{align*} 
\end{prop} 
\begin{proof} 
\par 
Let $Y_{\bul}$ be the \v{C}ech diagram of an affine open covering of $Y$ 
by 
finitely many affine open subschemes $U_j$'s of $Y$. 
Set $U^{(i)}_j:=Y^{(i)}_j\cap U_j$, 
$Y^{(i)}_0:=\coprod_jU^{(i)}_j$ and 
$Y^{(i)}_n:={\rm cosk}_0^{Y^{(i)}}(Y^{(i)}_0)_n$ $(n\in {\mab N})$. 
Let $Y_{\bul}\os{\sus}{\lo} {\cal P}_{\bul}$ be a closed immersion into 
a formally smooth formal scheme over ${\rm Spf}({\cal W}({\mab F}_q))$. 
Then we have a closed immersion $Y^{(i)}_0 \os{\sus}{\lo} \coprod^{(i)} {\cal P}_0$, 
where $\coprod^{(i)} {\cal P}_0$ is a finite sum of ${\cal P}_0$ 
which depends on $i$. 
Let $\del_j \col Y_0^{(i+1)}\lo Y_0^{(i)}$ $(0\leq j\leq i)$ 
be the standard face morphism. 
Then we have a natural morphism 
$\Del_j\col \coprod^{(i+1)} {\cal P}_0\lo \coprod^{(i)} {\cal P}_0$ 
fitting into the following commutative diagram 
\begin{equation*} 
\begin{CD}
Y_0^{(i+1)}@>{\del_j}>>Y_0^{(i)}\\
@V{\bigcap}VV @VV{\bigcap}V \\
\coprod^{(i+1)} {\cal P}_0@>{\Del_j}>>\coprod^{(i)} {\cal P}_0
\end{CD}
\end{equation*} 
and satisfying the standard relations. 
Set ${\cal P}^{(i)}_{\bul}:=
{\rm cosk}_0^{{\cal W}({\mab F}_q)}(\coprod^{(i)} {\cal P}_0)$.  
Let ${\rm sp}\col ]Y^{(i)}_{\bul}[_{{\cal P}^{(i)}_{\bul}}\lo Y^{(i)}_{\bul}$ 
be the specialization map. 
Then, as in \cite[(2.3)]{cric}, the following sequence 
\begin{align*} 
0\lo {\rm sp}_*(\Om^{\bul}_{]Y_{\bul}[_{{\cal P}_{\bul}}})\lo 
{\rm sp}_*(\Om^{\bul}_{]Y^{(0)}_{\bul}[_{{\cal P}^{(0)}_{\bul}}})\lo 
{\rm sp}_*(\Om^{\bul}_{]Y^{(1)}_{\bul}[_{{\cal P}^{(1)}_{\bul}}})\lo \cdots
\end{align*} 
is exact. 
Hence we have the following spectral sequence 
\begin{align*} 
E_1^{ij}=H^j_{\rm rig}(Y^{(i)}/K_0({\mab F}_q))\Lo H^{i+j}_{\rm rig}(Y/K_0({\mab F}_q)). 
\tag{5.1.2}\label{ali:sps} 
\end{align*} 
By (\ref{ali:trfrgh}) and this spectral sequence, 
we obtain the following formula: 
\begin{align*} 
[(E_p(Y/{\mab F}_q),F^*_q)]=
\sum_{i,j\geq 0}(-1)^{i+j}[H^j_{\rm rig}(Y^{(i)}/K_0({\mab F}_q)),F^*_q)]
\in {\cal K}(K_0({\mab F}_q)).  
\tag{5.1.3}\label{ali:zeht}
\end{align*}   
This formula implies (\ref{ali:zytp}). 
\end{proof}

\begin{coro}\label{coro:f}
Let $X/{\mab F}_q$ be a non-smooth 
combinatorial $K3$ surface {\rm (\cite{kul}, \cite{fs}, \cite{nlk3})}. 
$($We do not assume that $X$ has a log structure of simple normal crossing type.$)$
Let $m$ be the summation of the times of the processes of blowing downs making 
all irreducible components relatively minimal. 
Let $M_1$ $($resp.~$M_2)$ 
be the cardinality of the irreducible components of $\os{\circ}{X}$
whose relatively minimal models 
are ${\mab P}^2_{{\mab F}_q}$ $($resp.~Hirzeburch surfaces
=relatively minimal rational ruled surfaces$)$. 
Let $M$ be the cardinality of the irreducible components of $\os{\circ}{X}$. 
Then the following hold$:$ 
\par 
$(1)$ If $X$ is of ${\rm Type}$ ${\rm II}$ with double elliptic curve
$E/{\mab F}_q$, 
then 
\begin{equation*}
Z(\os{\circ}{X}/{\mab F}_q,t)= \dfrac{
{\rm det}(1-qtF^*_q \vert H^1_{\rm rig}(E/K_0({\mab F}_q)))^{M-2}
}{(1-t){\rm det}(1-tF^*_q \vert 
H^1_{\rm rig}(E/K_0({\mab F}_q)))(1-qt)^{M_1+2M_2+M-3+m}(1-q^2t)^M}. 
\tag{5.2.1}\label{eqn:kfmdad}
\end{equation*} 
\par 
$(2)$ Assume that $X$ is of 
${\rm Type}$ ${\rm III}$. 
Let $d$ be the cardinality of the double curves of $\os{\circ}{X}$. 
Then 
\begin{equation*}
Z(\os{\circ}{X}/{\mab F}_q,t)= \dfrac{1}{(1-t)^2(1-qt)^{M_1+2M_2+m-d}(1-q^2t)^M}. 
\tag{5.2.2}\label{eqn:kfdqad}
\end{equation*} 
\end{coro} 
\begin{proof} 
First we give a remark on the rigid cohomology of 
a smooth projective rational surface $S$ over ${\mab F}_q$. 
Set $S_{{\mab F}_{q^{n}}}:=S\us{{\mab F}_q}{\otimes}{{\mab F}}_{q^{n}}$ 
and $S_{\ol{\mab F}_q}:=S\us{{\mab F}_q}{\otimes}{\ol{\mab F}_q}$. 
\par 
Let $\ol{S}_{\rm min}$ be a relatively 
minimal model of $S_{\ol{\mab F}_q}$. 
If $\ol{S}_{\rm min}\simeq {\mab P}^2_{\ol{\mab F}_q}$,
we see that the motive 
$H(S_{\ol{\mab F}_{q}})$ is as follows by \cite[(6.12)]{dmi}:
\begin{align*} 
H(S_{\ol{\mab F}_{q}})\simeq 
H({\mab P}^2_{{\ol{\mab F}_{q}}}) \oplus H(D)(-1),
\end{align*} 
where $D$ is the disjoint sum of 0-dimensional points. 
Since $H^2({\mab P}^2_{\ol{\mab F}_q})$ 
is isomorphic to a Tate-twist, if a natural number $n$ is big enough, 
then $(F^*_q)^n$ on 
$H^2_{\rm rig}(S_{{\mab F}_{q^n}}/K_0({{\mab F}}_{q^{n}}))$ 
is ${\rm diag}(q^n, \ldots, q^n)$.
Hence the eigenvalues of $(F^*_q)^n$ are $q^n$ $(n \gg 0)$ and thus 
the eigenvalues of $F^*_q$ are $q$.
\par 
If $\ol{S}_{\rm min}$ is isomorphic to 
a relatively minimal ruled surface over a smooth curve 
$C$ over $\ol{\mab F}_q$,
the motive $H(S_{\ol{\mab F}_{q}})$ is as follows by 
\cite[(6.10), (6.12)]{dmi}:
\begin{align*} 
H(S_{\ol{\mab F}_q})\simeq H(C) 
\oplus H(C)(-1)\oplus H(D)(-1).
\end{align*}  
Hence we see that $F^*_q$ on $H^2_{\rm rig}(S/K_0({\mab F}_q))$ is 
${\rm diag}(q, \ldots, q)$ as above. 
\par 
(1):  
It is easy to check that 
\begin{equation*}
H^i_{\rm rig}(\os{\circ}{X}{}^{(0)}/K_0({\mab F}_q))= 
\begin{cases} 
K_0({\mab F}_q)^M& (i=0) \\
H^1_{\rm rig}(E/K_0({\mab F}_q))^{\oplus M-2} & (i=1)\\
K_0({\mab F}_q)(-1)^{M_1+2M_2+2(M-2)+m} & (i=2)\\
H^1_{\rm rig}(E/K_0({\mab F}_q))(-1)^{\oplus M-2}& (i=3)\\
K_0({\mab F}_q)(-2)^M& (i=4)
\end{cases} 
\label{eqn:kfda0d}
\end{equation*} 
and 
\begin{equation*}
H^i_{\rm rig}(\os{\circ}{X}{}^{(1)}/K_0({\mab F}_q))= 
\begin{cases} 
K_0({\mab F}_q)^{M-1}& (i=0) \\
H^1_{\rm rig}(E/K_0({\mab F}_q))^{\oplus M-1} & (i=1)\\
K_0({\mab F}_q)(-1)^{M-1}& (i=2). 
\end{cases} 
\label{eqn:kfdakd}
\end{equation*} 
Now (\ref{eqn:kfmdad}) follows from (\ref{ali:zytp}). 
\par 
(2): 
Let $T$ be the cardinality of the triple points of $\os{\circ}{X}$. 
It is easy to check that 
\begin{equation*}
H^i_{\rm rig}(\os{\circ}{X}{}^{(0)}/K_0({\mab F}_q))= 
\begin{cases} 
K_0({\mab F}_q)^M& (i=0) \\
0 & (i=1)\\
K_0({\mab F}_q)(-1)^{M_1+2M_2+m} & (i=2)\\
0& (i=3)\\
K_0({\mab F}_q)(-2)^M& (i=4), 
\end{cases} 
\label{eqn:kafdad}
\end{equation*} 
\begin{equation*}
H^i_{\rm rig}(\os{\circ}{X}{}^{(1)}/K_0({\mab F}_q))= 
\begin{cases} 
K_0({\mab F}_q)^d& (i=0) \\
0 & (i=1)\\
K_0({\mab F}_q)(-1)^{d}& (i=2)  
\end{cases} 
\label{eqn:kfdabd}
\end{equation*} 
and 
\begin{equation*}
H^0_{\rm rig}(\os{\circ}{X}{}^{(2)}/K_0({\mab F}_q))= 
K_0({\mab F}_q)^T. 
\label{eqn:kfdiad}
\end{equation*} 
Because the dual graph of $\os{\circ}{X}$ is a circle, 
$M-d+T=\chi({\mab S}^1)=2$. 
Now (\ref{eqn:kfdqad}) follows from (\ref{ali:zytp})
\end{proof} 

\begin{rema}\label{rema:tri}
If $p\not=2$, we can prove that $T$ is even (cf.~\cite{fs}).  
However we do not use this fact in this article. 
\end{rema}

\begin{coro}\label{coro:en}
Assume that $p\not=2$. 
Let $X/{\mab F}_q$ be a non-smooth combinatorial classical Enriques surface 
{\rm (\cite{kul}, \cite{nlk3})}. 
Let $M_1$, $M_2$,  $M$, $m$  and $d$ be as in {\rm (\ref{coro:f})}. 
Then the following hold$:$ 
\par 
$(1)$ If $X$ is of 
${\rm Type}$ ${\rm II}$ with double elliptic curve $E/{\mab F}_q$, 
then 
\begin{equation*}
Z(\os{\circ}{X}/{\mab F}_q,t)= \dfrac{{\rm det}(1-qtF^*_q \vert H^1_{\rm rig}(E/K_0({\mab F}_q)))^{M-1}}{(1-t)(1-qt)^{M_1+2M_2+M-1+m}(1-q^2t)^M}. 
\tag{5.4.1}\label{eqn:kfend}
\end{equation*} 
\par 
$(2)$ Assume that $X$ is of 
${\rm Type}$ ${\rm III}$. 
Let $d$ and $T$ be the cardinalities of the double curves of 
$\os{\circ}{X}$ and the triple points of $\os{\circ}{X}$, respectively. 
Then 
\begin{equation*}
Z(\os{\circ}{X}/{\mab F}_q,t)= 
\dfrac{1}{(1-t)(1-qt)^{M_1+2M_2+m-d}(1-q^2t)^M}. 
\tag{5.4.2}\label{eqn:kenad}
\end{equation*}
\end{coro} 
\begin{proof} 
(1): It is easy to check that 
\begin{equation*}
H^i_{\rm rig}(\os{\circ}{X}{}^{(0)}/K_0({\mab F}_q))= 
\begin{cases} 
K_0({\mab F}_q)^M& (i=0) \\
H^1_{\rm rig}(E/K_0({\mab F}_q))^{\oplus M-1} & (i=1)\\
K_0({\mab F}_q)(-1)^{M_1+2M_2+2(M-1)+m} & (i=2)\\
H^1_{\rm rig}(E/K_0({\mab F}_q))(-1)^{\oplus M-1}& (i=3)\\
K_0({\mab F}_q)(-2)^M& (i=4)
\end{cases} 
\label{eqn:kfdacd}
\end{equation*} 
and 
\begin{equation*}
H^i_{\rm rig}(\os{\circ}{X}{}^{(1)}/K_0({\mab F}_q))= 
\begin{cases} 
K_0({\mab F}_q)^{M-1}& (i=0) \\
H^1_{\rm rig}(E/K_0({\mab F}_q))^{\oplus M-1} & (i=1)\\
K_0({\mab F}_q)(-1)^{M-1}& (i=2). 
\end{cases} 
\label{eqn:kfdadd}
\end{equation*} 
\par 
(2): It is easy to check that 
\begin{equation*}
H^i_{\rm rig}(\os{\circ}{X}{}^{(0)}/K_0({\mab F}_q))= 
\begin{cases} 
K_0({\mab F}_q)^M& (i=0) \\
0 & (i=1)\\
K_0({\mab F}_q)(-1)^{M_1+2M_2+m} & (i=2)\\
0& (i=3)\\
K_0({\mab F}_q)(-2)^M& (i=4), 
\end{cases} 
\label{eqn:kfdaed}
\end{equation*} 
\begin{equation*}
H^i_{\rm rig}(\os{\circ}{X}{}^{(1)}/K_0({\mab F}_q))= 
\begin{cases} 
K_0({\mab F}_q)^d& (i=0), \\
0 & (i=1)\\
K_0({\mab F}_q)(-1)^{d}& (i=2) 
\end{cases} 
\label{eqn:kfdafd}
\end{equation*} 
and 
\begin{equation*}
H^0_{\rm rig}(\os{\circ}{X}{}^{(2)}/K_0({\mab F}_q))= 
K_0({\mab F}_q)^T. 
\label{eqn:kfdiabd}
\end{equation*} 
Because the dual graph of $\os{\circ}{X}$ is ${\mab P}^2({\mab R})$, 
$M-d+T=\chi({\mab P}^2({\mab R}))=1$. 
\end{proof}

\par 
Lastly we consider another type of local zeta functions. 
\par 
Let ${\cal V}$ be a complete discrete valuation ring of mixed characteristics with 
finite residue field ${\mab F}_q$ and let $K$ be the fraction field 
of ${\cal V}$. Let ${\mathfrak Y}$ be a proper smooth scheme over $K$
of dimension $d$ and let $I$ be the inertia group 
of the absolute Galois group ${\rm Gal}(\ol{K}/K)$.
Then the zeta function of ${\mathfrak Y}$ is defined as follows:
\begin{align*} 
Z({\mathfrak Y},t):=\Pi_{i=0}^{2d}{\rm det}(1-t\sigma \vert 
H^i_{\rm et}({\mathfrak Y}\us{K}{\otimes} \ol{K}, {\mab Q}_l)^I)^{(-1)^{h+1}},
\end{align*} 
where $\sigma \in {\rm Gal}(\ol{K}/K)$ is a lift of 
the geometric Frobenius of ${\rm Gal}(\ol{\mab F}_q/{\mab F}_q)$ 
and $l$ is a prime which is prime to $q$. If ${\mathfrak Y}$ is the 
generic fiber of a proper semistable family ${\cal Y}$ over ${\cal V}$ with 
special fiber $Y$, then the 
following formula holds by \cite{fkk} (\cite{ifkn}):
\begin{align*} 
Z({\mathfrak Y},t)=\Pi_{h=0}^{2d}{\rm det}(1-t\sigma \vert 
H^h_{{\rm ket}}(\ol{Y}, {\mab Q}_l)^I)^{(-1)^{i+1}}. 
\end{align*} 
\par 
Let ${\cal X}$ be a proper strict semistable family of 
surfaces over ${\cal V}$ with log special fiber $X$ over 
$s_{{\mab F}_q}$.

Then \cite[(6.3.3)]{msemi} tells us 
that  $Z({\cal X}_K,t)$
can be described by the log crystalline cohomologies by 
the coincidence of the monodromy filtration and the weight filtration 
(\cite[(8.3)]{nlpi}, \cite[(6.2.4)]{msemi}; however see 
\cite[(11.15)]{ndw} and \cite[(7.1)]{nlpi}.):
$$Z({\cal X}_K,t)=\Pi_{i=0}^{4}{\rm det}
(1-tF^*_q\vert (H^i_{\rm crys}(X/
{\cal W}(s_{{\mab F}_q}))_{K_0({\mab F}_q)})^{N=0})^{(-1)^{i+1}},$$
where ${\cal W}(s_{{\mab F}_q})$ is the canonical lift of $s_{{\mab F}_q}$ 
over ${\cal W}({\mab F}_q)$, 
$H^i_{\rm crys}(X/{\cal W}(s_{{\mab F}_q}))$ 
is the $i$-th log crystalline cohomology of $X/{\cal W}({\mab F}_q)$ 
and 
$$N\col H^i_{\rm crys}(X/
{\cal W}(s_{{\mab F}_q}))_{K_0({\mab F}_q)}\lo 
H^i_{\rm crys}(X/
{\cal W}(s_{{\mab F}_q}))_{K_0({\mab F}_q)}(-1)$$ 
is the $p$-adic monodromy operator. 
More generally, for a proper SNCL scheme 
$Y/s_{{\mab F}_q}$ of pure dimension $d$, 
set 
\begin{align*} 
Z(H^i(Y/K_0(s_{{\mab F}_q})),t):={\rm det}
(1-tF^*_q\vert 
(H^i_{\rm crys}(X/{\cal W}(s_{{\mab F}_q}))_{K_0({\mab F}_q)})^{N=0})^{(-1)^{i+1}}
\tag{5.4.3}\label{ali:dhxw}
\end{align*} 
and 
\begin{align*} 
Z(Y/s_{{\mab F}_q},t):=\prod_{i=0}^{2d}Z(H^i(Y/K_0(s_{{\mab F}_q})),t)^{(-1)^{i+1}} 
\tag{5.4.4}\label{ali:dhxzw}. 
\end{align*}

Let us recall the following result due to the author
(\cite[(8.3)]{nmw}, (cf.~\cite[(2.2)]{may}, \cite[(6.4)]{cla})): 

\begin{theo}[{\bf \cite[(8.3)]{nmw}}]\label{th:kumd}
Let $\kap$ be a perfect field of characteristic $p>0$. 
Let $X/s$ be an  SNCL $K3$ surface. 
Let $H^i_{{\rm log}}(X)$ $(i\in {\mab N})$ be the $i$-th log crystalline cohomology 
or the $i$-th Kummer \'{e}tale cohomology of $X/s$. 
Then the following hold$:$ 
\par
$(1)$ The $\star$-adic $(\,\star=p,l\,)$
monodromy filtration and the weight one 
on $H^i_{{\rm log}}(X)$ 
coincide.
\par
$(2)$  The following hold$:$
\par
$({\rm a})$ $X$ is of ${\rm Type}$  ${\rm I}$  if and only if $N=0$ on 
$H^2_{{\rm log}}(X)$.
\par
$({\rm b})$ $X$ is of ${\rm Type}$ ${\rm II}$ if and only if $N\not=0$ and 
$N^2=0$ on $H^2_{{\rm log}}(X)$.
\par
$({\rm c})$ $X$ is of ${\rm Type}$ ${\rm III}$ if and only if $N^2\not=0$  
on $H^2_{{\rm log}}(X)$. 
\end{theo}
\begin{proof} 
For the completeness of this article, we give the proof of (\ref{th:kumd}). 
\par 
We give the proof of this theorem in the $p$-adic case 
because the proof in the $l$-adic case 
is the same as that in the $p$-adic case. 
\par 
Recall the following weight spectral sequence (\cite[3.23]{msemi}, 
\cite[(2.0.1)]{ndw}): 
\begin{align*} 
E_{1}^{-k, i+k}&= 
\us{j\geq {\rm max}\{-k, 0\}}
{\bigoplus}
H^{i-2j-k}_{\rm rig}
(\os{\circ}{X}{}^{(2j+k)}/{K_0})(-j-k)\Lo 
H^i_{\rm crys}(X/{\cal W}(s))_{K_0}.
\tag{5.5.1}\label{eqn:pstwtsp}
\end{align*}
(See \cite{ndw} for the mistakes in \cite{msemi}.) 
Here we have used Berthelot's comparison isomorphism 
$H^{i}_{\rm crys}(Y/{\cal W})_{K_0}=H^{i}_{\rm rig}(Y/{K_0})$ $(i\in {\mab N})$ 
for a proper smooth scheme $Y$ over $\kap$. 
By \cite[(3.6)]{ndw} this spectral sequence degenerates at $E_2$. 
(The $l$-adic analogue of this spectral sequence also degenerates at 
$E_2$ by Nakayama's theorem (\cite[(2.1)]{nd}).) 
\par
(1): We may assume that $\kap$ is algebraically closed.
If $X/s$ is of Type I, there is nothing to prove.
\par
If $X$ is of Type III, the double curves and 
the irreducible components are rational, and hence 
$E_1^{0,1}=E_1^{1,1}=E_1^{0,3}=E_1^{-1,3}=0$. 
By \cite[(3.5) 3)]{nlk3}, 
$H^1_{{\rm log}{\textrm -}{\rm crys}}(X/{\cal W})=0$ and hence 
we have $E_2^{-1,2}=0$. 
(Note that we also have the similar vanishing for the first Kummer \'{e}tale cohomology of $X$ by the vanishing above and the existence of 
the ${\mab Q}$-structure of $E_2^{-1,2}$ 
(cf.~the proof of \cite[(8.3)]{nlpi}). 
By taking the duality in \cite[(10.5)]{ndw}, 
$E_2^{1,2}=0$. 
By \cite[6.2.1]{msemi} the $p$-adic monodromy operator 
$N\col H^2_{\rm crys}(X/{\cal W}(s))
\lo 
H^2_{\rm crys}(X/{\cal W}(s))(-1)$ 
induces an isomorphism 
$N^2 \col E_2^{-2,4}\os{\simeq}{\lo}E_2^{2,0}(-2)=K_0$. 
\par
If $X$ is of Type II, $E_{1}^{-2,4}=E_{1}^{2,0}=0$. 
By \cite[(3.5) 3)]{nlk3} again,
$H^1_{{\rm log}{\textrm -}{\rm crys}}(X/{\cal W})=
H^3_{{\rm log}{\textrm -}{\rm crys}}(X/{\cal W})=0$. 
Hence $E_{2}^{ij}=0$ for $i+j=1,3$.
Because $N\col H^2_{\rm crys}(X/{\cal W}(s))
\lo 
H^2_{\rm crys}(X/{\cal W}(s))(-1)$ 
induces an isomorphism $E_2^{-1,3}
\os{\simeq}{\lo}E_2^{1,1}(-1)$ by \cite[6.2.2]{msemi}, we have proved (1).
\parno
(2): (2) follows from (1) and the non-vanishings of $E_2^{1,1}$ 
in the Type II case and $E_2^{2,0}$ in the Type III case, respectively.
\end{proof}

\begin{rema}
The author has found the theorem (\ref{th:kumd}) in December 1996 
by using the $p$-adic weight spectral sequence 
(\ref{eqn:pstwtsp}). 
The key point of the proof is 
to notice to use the $p$-adic weight spectral sequence  
of $X/s$ instead of the Clemens-Schmid exact sequence used in 
Kulikov's article \cite{kul}. 
(In fact, the complex analogue (\ref{theo:cct}) below of 
(\ref{th:kumd}) holds; this is a generalization of Kulikov's theorem in [loc.~cit.]
and the proof of (\ref{th:kumd}) is simpler than that in [loc.~cit.].
To my surprise, mathematicians who are working over ${\mab C}$ 
have not used the weight spectral sequence (\ref{eqn:dfwqtfi}).)
The author has finished writing the preprint \cite{nmw} by 2000 
at the latest (cf.~\cite[Remark 2.4 (3)]{nd}). 
However, after that, he has noticed that there are too many non-minor mistakes in 
theory of log de Rham-Witt complexes in Hyodo-Kato's article \cite{hk} 
and Mokrane's article \cite{msemi} 
as pointed out in \cite{ndw}. 
Because he has used Hyodo-Kato's and Mokrane's 
theory in \cite{nmw} heavily, he has to use their results in correct ways. 
However he has used his too much time for correcting their results in \cite{ndw}, 
he has no will to publish \cite{nmw} now (because \cite{nmw} is quite long and 
because he has to use more time for adding comments about 
Hyodo-Kato's and Mokrane's articles in \cite{nmw}). 
For example, $\nu$ is in  \cite{msemi} 
is {\it not} a morphism of complexes, the left $N$ in the diagram in \cite[(2.2)]{may}
is incorrect. 
\par 
In \cite[(2.2)]{may} Matsumoto has proved (\ref{th:kumd}) 
for semistable algebraic spaces of $K3$-surfaces 
after looking at the proof in \cite{nmw}. 
(See ``Proof of $p$-adic case'' in the proof of \cite[Proposition 2.2]{may}.)
\end{rema}

\begin{theo}[{\bf cf.~\cite{kul}}]\label{theo:cct}
Let $s$ be the log point of ${\mab C}$. 
Let $X/s$ be an analytic SNCL $K3$ surface. 
Let $X_{\infty}$ be the base change of 
the Kato-Nakayama space $X^{\log}$ of $X$ {\rm (\cite{kn})} with respect to 
the morphism ${\mab R}\owns x\lom {\rm exp}(2\pi{\sqrt{-1}})\in {\mab S}^1$. 
Let $N\col H^i(X_{\infty},{\mab Q})\lo H^i(X_{\infty},{\mab Q})(-1)$ 
$(i\in {\mab N})$ 
be the monodromy operator constructed in {\rm \cite{fn}}. 
Then the following hold$:$ 
\par
$(1)$ The weight filtration on $H^i(X_{\infty},{\mab Q})$ 
constructed in {\rm \cite{fn}} 
coincide with the monodromy filtration on $H^i(X_{\infty},{\mab Q})$
\par
$(2)$  The following hold$:$
\par
$({\rm a})$ $X$ is of ${\rm Type}$  ${\rm I}$  if and only if $N=0$ on 
$H^2(X_{\infty},{\mab Q})$.
\par
$({\rm b})$ $X$ is of ${\rm Type}$ ${\rm II}$ if and only if $N\not=0$ and 
$N^2=0$ on $H^2(X_{\infty},{\mab Q})$.
\par
$({\rm c})$ $X$ is of ${\rm Type}$ ${\rm III}$ if and only if $N^2\not=0$  
on $H^2(X_{\infty},{\mab Q})$. 
\end{theo} 
\begin{proof} 
By \cite[(2.1.10)]{nlpi} 
we have the following weight spectral sequence: 
\begin{equation*}
E_{1,\inf}^{-k, h+k}= 
\us{j\geq {\rm max}\{-k, 0\}}{\bigoplus}
H^{h-2j-k}(\os{\circ}{X}{}^{(2j+k+1)},{\mab Q})(-j-k) \Lo 
H^h(X_{\infty},{\mab Q}). 
\tag{5.7.1}\label{eqn:dfwqtfi}
\end{equation*} 
By \cite[(5.9)]{fr}, 
if $X$ is a  combinatorial ${\rm Type}$ ${\rm II}$ or ${\rm Type}$ ${\rm III}$ 
$K3$ surface over ${\mab C}$, 
then $H^0(X, \Om_{X/{\mab C}}^{1})=0$.
(Of course, if $X$ is of ${\rm Type}$ ${\rm I}$, 
then $H^0(X, \Om_{X/{\mab C}}^{1})=0$ by Hodge symmetry.)
Hence $H^1(X_{\infty},{\mab C})=H^1_{\rm dR}(X/{\mab C})
=H^0(X,\Om^1_{X/s})\oplus H^1(X,{\cal O}_X)=0$. 
Here we have used the isomorphism 
between Steenbrink complexes $A_{\mab Q}\otimes_{\mab Q}{\mab C}$ 
and $A_{\mab C}$ of $X$ 
and the isomorphism between 
$A_{\mab C}$ and $\Om^{\bul}_{X/s}$ (\cite{fn}). 
By the duality of the $E_2$-terms of (\ref{eqn:dfwqtfi}) 
(\cite[(5.15) (2)]{nlpi}) and the degeneration at 
$E_2$ of (\ref{eqn:dfwqtfi}) (by Hodge theory), 
we obtain the vanishing of $H^3(X_{\infty},{\mab C})$. 
The rest of the proof is the same as that of (\ref{th:kumd}). 
\end{proof}

\begin{theo}[{\bf \cite[(15.1)]{nmw}}]\label{theo:zt}
Let $X/s_{{\mab F}_q}$ be a projective SNCL $K3$ surface. 
Then the following hold$:$
\par 
$(1)$ 
\begin{equation*}
Z(H^i(X/K_0(s_{{\mab F}_q})),t)= 
\begin{cases} 
1-t& (i=0) \\
1 & (i=1,3)\\
1-q^2t& (i=4).   
\end{cases} 
\label{eqn:kfdagd}
\end{equation*} 
\par
$(2)$ If $X$ is of 
${\rm Type}$ ${\rm II}$  with double elliptic curve
$E$, then 
\begin{align*} 
Z(H^2(X/K_0(s_{{\mab F}_q})),t)={\rm det}(1-tF^*_q \vert H^1_{\rm rig}(E/K_0))(1-qt)^{18}.
\end{align*} 
Consequently 
\begin{align*} 
Z(X/s_{{\mab F}_q},t)=\dfrac{1}{(1-t)
{\rm det}(1-tF^*_q \vert H^1_{\rm rig}(E/K_0))(1-qt)^{18}(1-q^2t)}.
\end{align*} 
\par
$(3)$ If $X$ is of 
${\rm Type}$ ${\rm III}$, then 
\begin{align*} 
Z(H^2(X/K_0(s_{{\mab F}_q})),t)=(1-t)(1-qt)^{19}.
\end{align*} 
Consequently 
\begin{align*} 
Z(X/s_{{\mab F}_q},t)=\dfrac{1}{(1-t)^2(1-qt)^{19}(1-q^2t)}.
\end{align*} 
\end{theo}
\begin{proof}
By \cite[(3.5)]{nlk3}, $H^i_{\rm crys}(X/{\cal W}(s_{{\mab F}_q}))=0$ $(i=1,3)$.  
Thus $Z(H^i(X/K_0(s_{{\mab F}_q})),t)=1$ $(i=1,3)$. 
By \cite[(6.9)]{nlk3}, $X$ is the log special fiber of 
a projective semistable family ${\cal X}$ over 
${\rm Spec}(W({\mab F}_q))$. 
By \cite[(6.10)]{nlk3}, the generic fiber of ${\cal X}$ is a K3 surface. 
Hence, by Hyodo-Kato's isomorphism (\cite[(5.1)]{hk}) 
(however see \cite[\S7]{ndw} for incompleteness of 
the proof of Hyodo-Kato isomorphism), 
${\rm dim}_{K_0({\mab F}_q)}
H^2_{\rm crys}(X/{\cal W}(s_{{\mab F}_q}))_{K_0({\mab F}_q)}=22$. 
\par 
(1): In this case, by (\ref{th:kumd}), $N\not=0$, $N\col E_2^{-1,2}\lo E_2^{1,1}$ 
is an isomorphism,  $N^2=0$ on 
$H^2_{\rm crys}(X/W(s_{{\mab F}_q}))_{K_0({\mab F}_q)}$ 
and $E_2^{-2,4}=E_2^{2,0}=0$.  
Hence we have the following exact sequence by (\ref{th:kumd}): 
$$ 0 \lo E_2^{1,1} \lo {\rm Ker}(N) \lo E_2^{0,2}\lo 0.$$
Because $E_2^{1,1}\simeq H^1_{\rm rig}(E/K_0({\mab F}_q))$, 
${\rm det}(1-tF^*_q \vert E_2^{1,1})=
{\rm det}(1-tF^*_q \vert H^1_{\rm rig}(E/K_0({\mab F}_q))$. 
On the other hand, $E_2^{0,2}$ is a 
subquotient of 
$$H^0_{\rm rig}(X^{(2)}/K_0({\mab F}_q))(-1)
\oplus H^2_{\rm rig}(X^{(1)}/K_0({\mab F}_q)).
$$ 
Hence $F^*_q$ on $E_2^{0,2}$ is 
${\rm diag}(q, \ldots, q)$ as shown in the proof of (\ref{coro:f}). 
Since ${\rm Ker}(N)$ is 20-dimensional, we obtain (2).
\parno
(2): In this case, by (\ref{th:kumd}), $N^2\not=0$, $N^3=0$ and
$E_2^{-1,3}=E_2^{1,1}=0$. 
Because $N^2 \col E_2^{-2,4} {\lo} E_2^{2,0}(-2)$ is 
an isomorphism, 
$N \col  E_2^{0,2} {\lo} E_2^{2,0}(-1)$ is surjective and
hence the kernel of $N$ is 
$20$-dimensional.
Obviously $F^*_q={\rm id}$ on  $E_2^{2,0}$. 
As in (1), $F^*_q$ on $E_2^{0,2}$ is 
${\rm diag}(q, \ldots, q)$. Hence we obtain (2).
\end{proof}

\begin{theo}[{\bf \cite[(15.2)]{nmw}}]\label{theo:zeet}
Let $X/s_{{\mab F}_q}$ be 
a projective non-smooth SNCL classical Enriques surface. 
Then 
\begin{equation*}
Z(H^i(X/K_0(s_{{\mab F}_q})),t)= 
\begin{cases} 
1-t& (i=0), \\
1& (i=1,3), \\
(1-qt)^{10} & (i=2)\\
1-q^2t& (i=4).   
\end{cases} 
\label{eqn:kfdahd}
\end{equation*} 
Consequently 
\begin{align*} 
Z(X/s_{{\mab F}_q},t)=\dfrac{1}{(1-t)(1-qt)^{10}(1-q^2t)}.
\end{align*} 
\end{theo}
\begin{proof} 
By \cite[(7.1)]{nlk3},  
$H^i_{\rm crys}(X/{\cal W}(s_{{\mab F}_q}))=0$ $(i=1,3)$ 
and hence $Z(H^i(X/K_0(s_{{\mab F}_q})),t)=1$ $(i=1,3)$. 
By \cite[(7.1)]{nlk3} and the argument in \cite[(6.8), (6.11)]{nlk3}, 
$X$ is the log special fiber of a projective semistable family 
${\cal X}$ over ${\cal W}({\mab F}_q)$ and 
the generic fiber of ${\cal X}$ is a classical Enriques surface. 
Hence ${\rm dim}_{K_0({\mab F}_q)}
H^2_{\rm crys}(X/{\cal W}(s_{{\mab F}_q}))_{K_0 ({\mab F}_q)}=10$. 
The rest of the proof is the same as that of (\ref{theo:zt}) by noting that 
$0=E_2^{-1,3}=E_2^{11}=E_2^{20}=E_2^{-2,4}$, 
where $E_2^{\bul \bul}$'s are 
$E_2$-terms of the spectral sequence (\ref{eqn:pstwtsp}). 
\end{proof} 

\parno
\begin{center}
{{\rm \Large{\bf Appendix}}}
\end{center}

\section{A remark on Katsura and Van der Geer's result}\label{rema:arkg} 
In this section we generalize the argument in the proof of (\ref{theo:xfh}) (3).
\par 
First we recall the following theorem in \cite{ny}. 
This is a generalization of Katsura and Van der Geer's theorem
(\cite[(5.1), (5.2), (16.4)]{vgk}).   

\begin{theo}[{\bf \cite[(2.3)]{ny}}]\label{prop:nex} 
Let $\kap$ be a perfect field of characteristic $p>0$. 
Let $Y$ be a proper scheme over $\kap$.
$($We do not assume that $Y$ is smooth over $\kap$.$)$ 
Let $q$ be a nonnegative integer. 
Assume that $H^q(Y,{\cal O}_Y)\simeq \kap$,  
that $H^{q+1}(Y,{\cal O}_Y)=0$ and 
that $\Phi^q_{Y/\kap}$ is pro-representable. 
Assume also 
that the Bockstein operator 
\begin{align*} 
\bet \col H^{q-1}(Y,{\cal O}_Y)\lo 
H^q(Y,{\cal W}_{n-1}({\cal O}_Y))
\end{align*} 
arising from the following exact sequence 
\begin{align*} 
0\lo {\cal W}_{n-1}({\cal O}_Y)\os{V}{\lo} {\cal W}_n({\cal O}_Y)
{\lo} {\cal O}_Y\lo 0
\end{align*} 
is zero for any $n\in {\mab Z}_{\geq 2}$.   
Let $V\col  {\cal W}_{n-1}({\cal O}_Y) \lo {\cal W}_n({\cal O}_Y)$ 
be the Verschiebung morphism and 
let $F\col {\cal W}_{n}({\cal O}_Y) \lo {\cal W}_n({\cal O}_Y)$ 
be the induced morphism by the Frobenius endomorphism of ${\cal W}_{n}(Y)$. 
Let $n^q(Y)$ be the minimum of positive integers $n$'s  
such that the induced morphism 
$$F\col H^q(Y,{\cal W}_n({\cal O}_Y))\lo H^q(Y,{\cal W}_n({\cal O}_Y))$$ 
by the $F\col {\cal W}_{n}({\cal O}_Y) \lo {\cal W}_n({\cal O}_Y)$ is not zero. 
$($If $F=0$ for all $n$, then set $n^q(Y):=\infty.)$ 
Let $h^q(Y/\kap)$ be the height of 
the Artin-Mazur formal group $\Phi^q_{Y/\kap}$ of $Y/\kap$. 
Then $h^q(Y/\kap)=n^q(Y)$.
\end{theo}

\begin{prop}\label{prop:nnz}
Let the notations be as in {\rm (\ref{prop:nex})}.  
Let $D(\kap)$ be the Cartier-Dieudonn\'{e} algebra over $\kap$. 
Then the following hold$:$ 
\par 
$(1)$ ${\rm length}_{\cal W} H^q(Y,{\cal W}_n({\cal O}_Y))=n$ $(n\in {\mab Z}_{\geq 1})$. 
\par
$(2)$ Set $h:=h(\Phi^q_{Y/\kap})$. Assume that $h<\infty$. 
Let us consider the following natural surjective morphism 
$H^q(Y,{\cal W}({\cal O}_Y))\lo H^q(Y,{\cal W}_h({\cal O}_Y))$. 
Then this morphism induces the following isomorphism
\begin{align*} 
H^q(Y,{\cal W}({\cal O}_Y))/p\os{\sim}{\lo} H^q(Y,{\cal W}_h({\cal O}_Y))
\tag{6.2.1}\label{ali:oyyh}
\end{align*} 
of $D(\kap)/p$-modules.
\end{prop} 
\begin{proof}
(1): By the assumptions 
we have the following exact sequence 
\begin{align*} 
0\lo H^q(Y,{\cal W}_{n-1}({\cal O}_Y))\os{V}{\lo} 
H^q(Y,{\cal W}_n({\cal O}_Y))\lo 
H^q(Y,{\cal O}_Y)\lo 0. 
\end{align*} 
(1) immediately follows from this. 
\par 
(2): First assume that $h=1$. Then 
$H^q(Y,{\cal W}({\cal O}_Y))\simeq {\cal W}$. 
In this case, (2) is obvious. 
\par 
Next assume that $1<h<\infty$.  
Set $M_n:=H^q(X,{\cal W}_n({\cal O}_X))$ $(n\in {\mab Z}_{\geq 1})$ 
and $M:=H^q(X,{\cal W}({\cal O}_X))$. 
Consider the following exact sequence 
\begin{align*}
0\lo {\cal W}_{m-n}({\cal O}_Y)\os{V^n}{\lo}  {\cal W}_{m}({\cal O}_Y)\lo 
{\cal W}_n({\cal O}_Y)\lo 0
\end{align*}
for $m>n$. 
By the assumption and (\ref{ali:yexyk}) we see that 
$H^{q+1}(Y,{\cal W}_m({\cal O}_X))=0$ for any $m$. 
Hence the natural morphism $M_m\lo M_n$ is surjective
and consequently  
the natural morphism $M\lo M_n$ is surjective.  
In particular, the natural morphism 
$M\lo M_h$ is surjective. 
Let $\eta$ be an element of $M_h$. 
We claim that $p\eta=0$. 
\par 
We have to distinguish the operator 
$F\col M_n\lo M_n$ and the operator $F\col M_n\lo M_{n-1}$. 
The latter ``$F$ is equal to $R_nF$, where 
$R_n\col M_n\lo M_{n-1}$ is the projection. 
We denote $R_nF$ by $F_n$ to distinguish two $F$'s.  
Since the following diagram 
\begin{equation*} 
\begin{CD} 
M_h@>{R_h}>> M_{h-1}\\
@V{F}VV @VV{F}V \\
M_h @>{R_h}>> M_{h-1}
\end{CD} 
\tag{6.2.2}\label{cd:hhr1} 
\end{equation*} 
is commutative, we have the following: 
\begin{align*} 
p\eta=VF_h(\eta)=VR_hF(\eta)=VFR_h(\eta).
\end{align*}  
Since $F=0$ on $M_{h-1}$ by (\ref{prop:nex}), the last term is equal to zero.  
Hence $p\eta=0$.  
Consequently the natural morphism 
$H^q(Y,{\cal W}({\cal O}_Y))\lo H^q(Y,{\cal W}_h({\cal O}_Y))$ 
factors through the projection 
$H^q(Y,{\cal W}({\cal O}_Y))\lo H^q(Y,{\cal W}({\cal O}_Y))/p$. 
Since the morphism (\ref{ali:oyyh}) is surjective and 
${\dim}_{\kap}H^q(Y,{\cal W}({\cal O}_Y))/p=h=
{\dim}_{\kap}H^q(Y,{\cal W}_h({\cal O}_Y))$ by (1), 
the morphism (\ref{ali:oyyh}) is an isomorphism. 
\end{proof} 

\begin{rema} 
Let the notations be as in (\ref{theo:xfh}) (2).
By using only (\ref{prop:nnz}) for the case $q=d$, 
we can prove that 
\begin{equation*}
\# Y({\mab F}_{q^k}) \equiv1~{\rm mod}~ p^{[\frac{ke+1}{2}]} 
\quad (k\in {\mab Z}_{\geq 1}),
\tag{6.3.1}\label{eqn:kfda2td}
\end{equation*}
where $[~]$ is the Gauss symbol. 
However the congruence (\ref{eqn:kfda2td}) is not sharper than 
(\ref{eqn:kfd2td}); only in the case $h=2$,  (\ref{eqn:kfda2td}) 
is equivalent to (\ref{eqn:kfd2td}). 
\end{rema}


\bigskip
\bigskip
\parno
Yukiyoshi Nakkajima 
\parno
Department of Mathematics,
Tokyo Denki University,
5 Asahi-cho Senju Adachi-ku,
Tokyo 120--8551, Japan. 
\parno
{\it E-mail address\/}: 
nakayuki@cck.dendai.ac.jp

\end{document}